\documentclass[10pt,a4paper, leqno]{article}
\usepackage[utf8]{inputenc}
\usepackage{amsmath}
\usepackage{amsfonts}
\usepackage{amssymb}
\usepackage{comment}
\usepackage[hyperfootnotes=false]{hyperref}
\usepackage{enumitem,pstricks,xy,pst-node}
\xyoption{all}
\usepackage{amsthm}
\usepackage{booktabs} 
\usepackage[color]{changebar} 
\cbcolor{red} 
\usepackage{mathrsfs} 
\usepackage{verbatim} 
\usepackage[disable] 
{todonotes}
\usepackage[normalem]{ulem} 
\usepackage{caption} 
\usepackage{nicefrac} 
\usepackage{enumitem} 

\newtheorem{theorem}{Theorem}[section]
\newtheorem{lemma}[theorem]{Lemma}
\newtheorem{proposition}[theorem]{Proposition}

\newtheorem{definition}[theorem]{Definition}
\theoremstyle{definition}
\newtheorem{formula}[theorem]{Formula}
\newtheorem{remark}[theorem]{Remark}
\newtheorem*{ack}{Acknowledgements}
\newtheorem*{fact}{Fact}
\theoremstyle{remark}
\newtheorem{example}[theorem]{Example}
\newtheorem{examples}[theorem]{Examples}

\newcommand{\PP}{\mathbb{P}}
\newcommand{\Mac}{{\texttt {Macaulay2}}}
\newcommand{\Gr}{\mathbf{Gr}}
\newcommand{\OGr}{\mathbf{OGr}}
\newcommand{\IGr}{\mathbf{IGr}}
\newcommand{\HOM}{\mathcal{H} om}

\def\cT{{\mathcal T}}
\def\cU{{\mathcal U}}
\def\zero{\mathscr{Z}}
\def\Grb{\mathbf{Gr}}

\def\PP{\mathbf P}

\def\CC{\mathbb{C}}

\def\QQ{\mathbb{Q}}
\def\ZZ{\mathbb{Z}}

\def\cE{{\mathcal E}}
\def\cH{{\mathcal H}}

\def\cV{{\mathcal V}}
\def\cW{{\mathcal W}}
\def\cO{{\mathcal O}}
\def\cF{{\mathcal F}}
\def\cQ{{\mathcal Q}}
\def\cL{{\mathcal L}}
\def\ra{\rightarrow}
\def\lra{\longrightarrow}
\def\fg{\mathfrak{g}}
\def\fh{\mathfrak{h}}

\def\fp{\mathfrak{p}}
\def\fsl{\mathfrak{sl}}
\def\fe{\mathfrak{e}}

\def\af1{\mathbf{aff}_1}

\DeclareMathOperator{\im}{Im}
\DeclareMathOperator{\rank}{rank}
\DeclareMathOperator{\td}{td}
\DeclareMathOperator{\Aut}{Aut}
\DeclareMathOperator{\End}{End}
\DeclareMathOperator{\Hom}{Hom}
\DeclareMathOperator{\hhh}{h}

\DeclareMathOperator{\ctop}{c_{top}}
\DeclareMathOperator{\Sym}{Sym}
\DeclareMathOperator{\Sing}{Sing}
\DeclareMathOperator{\codim}{codim}

\DeclareMathOperator{\rad}{rad}
\DeclareMathOperator{\ad}{ad}
\DeclareMathOperator{\ind}{index}

\DeclareMathOperator{\HHH}{H}
\DeclareMathOperator{\RRR}{R}

\newcommand{\ladi}{\begin{lastadd}}
\newcommand{\ladf}{\end{lastadd}}
\newcommand{\lrei}{\begin{lastrem}}
\newcommand{\lref}{\end{lastrem}}
\newenvironment{lastadd}
{\cbstart\color{red}}
{\todo{red to remove}\cbend}
\newenvironment{lastrem}
{\cbstart\color{yellow}}
{\cbend}

\newcommand\blfootnote[1]{%
	\begingroup
	\renewcommand\thefootnote{}\footnote{#1}%
	\addtocounter{footnote}{-1}%
	\endgroup
}

\author{Vladimiro Benedetti \and Sara Angela Filippini \and Laurent Manivel \and Fabio Tanturri}
\title{Orbital degeneracy loci and applications}
\begin{document}
\date{}


\AtEndDocument{\noindent Vladimiro Benedetti\\
	D\'epartement de Math\'ematiques et Applications\\ENS, PSL Research University\\75230 Paris CEDEX 05, France.\\
	Vladimiro.Benedetti@ens.fr\\
	\rule{1pt}{0pt}\\
	Sara Angela Filippini\\
	Department of Mathematics, Imperial College London\\
	South Kensington Campus\\
	London SW7 2AZ, UK.\\
	s.filippini@imperial.ac.uk\\
	\rule{1pt}{0pt}\\
	Laurent Manivel\\
	Institut de Math\'ematiques de Toulouse, UMR 5219\\
	Universit\'e de Toulouse, CNRS, UPS IMT\\
	F-31062 Toulouse Cedex 9, France.\\
	manivel@math.cnrs.fr\\
	\rule{1pt}{0pt}\\
	Fabio Tanturri\\
	Laboratoire Paul Painlev\'e, UMR CNRS 8524\\
	Universit\'e de Lille\\
	59655 Villeneuve d'Ascq CEDEX, France.\\
	Fabio.Tanturri@math.univ-lille1.fr}
\maketitle

\begin{abstract}
Degeneracy loci of morphisms between vector bundles have been used in a wide variety of situations. 
We introduce a vast generalization of this notion, based on orbit closures of algebraic groups in their
linear representations. A preferred class of our orbital degeneracy loci is characterized by a certain 
crepancy condition on the orbit closure, that allows to get some control on the canonical sheaf. This 
condition is fulfilled for Richardson nilpotent orbits, and also for partially decomposable skew-symmetric 
three-forms in six variables. In order to illustrate the efficiency and flexibility of our methods, 
we construct in both situations many Calabi--Yau manifolds of dimension three and four, as well as a few
Fano varieties, including some new Fano fourfolds.
\end{abstract}
\setcounter{tocdepth}{1}
\tableofcontents

\blfootnote{2010 MSC Primary: 14N05; Secondary: 14E15, 14J32, 14J45, 14M12.}

\section{Introduction}

Algebraic geometry is full of amazing abstract statements about varieties and schemes. Sometimes one can feel a bit frustrated 
about the relatively small number of interesting varieties or schemes that we are able to effectively construct. As  Simpson 
formulates it in a slightly different context \cite{simpson04}, {\it we have the 
	impression that there is a huge mass of stuff out there, waiting to be constructed or seen, but we have no idea how to get there}.

Calabi--Yau threefolds are probably a good example: even though 
huge databases have been constructed, which essentially compile
complete intersections in toric varieties, our feeling is that 
there is still {\it a huge mass of stuff} to be discovered, 
consisting  of  Calabi--Yau threefolds of very different types. 
The situation is even more frustrating as far as compact hyperk\"ahler
manifolds are concerned: a few beautiful constructions have been
known for some time, but even if we can imagine that {\it there is
	some stuff out there, waiting to be constructed or seen,  
	we have no idea how to get there}. In fact no new  hyperk\"ahler
manifold has been constructed in this century. 

The purpose of this paper is to introduce some basic techniques 
that should enrich our toolbox, and show how to effectively construct interesting varieties using these techniques. The methods we introduce are rather flexible. The thread we decided to follow in order to illustrate their efficiency was to construct 
varieties with trivial canonical bundle in low dimension, essentially threefolds and fourfolds.  
Our hope was of course to 
discover some new  hyperk\"ahler fourfolds, or at least some 
new explicit constructions of polarized hyperk\"ahler fourfolds.
For the time being this has not happened, but we sincerely hope that other, more astute mathematicians will be able to use our techniques and fulfill this goal. 

Our initial motivation was to generalize the very classical
notion of degeneracy loci of morphisms between vector bundles.
The starting point of our project was the observation that 
the universal models of degeneracy loci are just the spaces 
of matrices of a given format, of rank bounded by a given integer. Those spaces are exactly the orbit closures of the linear groups acting as usual on the space of matrices. From this point of view, they are just a basic series of examples inside the world of 
representations of algebraic groups with only finitely many 
orbits. Irreducible representations of complex reductive groups
with this property were classified by V.\ Kac in a very influential paper \cite[Theorem 2]{Kac80}. There are many interesting cases, some of them very classical, other ones related to exceptional groups and still rather mysterious; but we have accumulated  a huge amount of information about those orbits, 
which are in themselves extremely interesting varieties. 

Beyond orbit closures, we can more generally consider an invariant closed subvariety  inside some linear representation of an algebraic group. This is the starting point for defining our orbital degeneracy loci, which are nothing else than relative versions of these invariant subvarieties, just as degeneracy loci of morphisms between vector bundles are relative versions of varieties of matrices with bounded rank. In fact the construction has nothing to do with the finiteness of orbit closures, and has a huge flexibility. But the most favorable situation happens when the subvariety is defined by a \emph{Kempf collapsing} satisfying a particular crepancy condition: in such a case, the relative version of the collapsing allows us to control the canonical sheaf of our degeneracy loci. We will focus on two situations for which this crepancy condition is fulfilled.

The first one is provided by skew-symmetric three-forms in six variables that are partially decomposable. The second one corresponds to nilpotent orbit closures, more precisely the so-called \emph{Richardson} ones, for which we have resolutions (or alterations) of singularities given by a Kempf collapsing similar to the famous Springer resolution. For both of these situations, we will use the relative version of the collapsing to construct examples of special varieties; typically, we will need to find, for our base variety, Fano varieties of a given dimension and a given index endowed with a suitable vector bundle or, more generally, a suitable principal bundle.

One of the limitations of our methods is that we 
have little understanding (and  only few  constructions) of vector bundles on Fano manifolds of higher dimension, but this 
understanding is likely to improve in the future. At present, we take advantage of the well-known fact that most 
of the Fano varieties of large index we have at our disposal are constructed from Grassmannians or other rational
homogeneous spaces, which have the nice property of being endowed with homogeneous vector bundles. Using 
those, we are able to construct  several  families of Calabi--Yau threefolds and  many families of Calabi--Yau fourfolds, as well as several examples of Fano varieties.
We hope this will convince our readers that our methods are really efficient, and that they have the 
potential for being applied in different contexts as well.

\medskip
The structure of the paper is the following. In Section \ref{geomtecnfor}, we define an orbital degeneracy locus, explain how to use a Kempf collapsing to control its canonical sheaf, and give a first series of relevant examples. In Section \ref{partDecForms}, we concentrate on three-forms in six variables; we explain  how they allow to 
construct threefolds and fourfolds with trivial canonical bundle starting from a suitable rank six vector bundle
on a Fano manifold of dimension eight or nine and index five; we give lists of explicit varieties and vector bundles satisfying all the required conditions. Section \ref{nilpotOrbits} focuses on nilpotent orbit closures; we explain  how each 
Richardson orbit can be used to construct threefolds and fourfolds with trivial canonical bundle, starting from 
a Fano manifold of suitable dimension and index, and we provide lists of explicit examples. In Section 
\ref{fanoDegLoci} we adapt our techniques in order to produce Fano or almost Fano manifolds, which is also an interesting problem; we describe  the (almost) Fano threefolds we are able to construct, and we identify them explicitly using the existing classifications. 

In Appendix \ref{appendix} we explain how we computed some of the invariants of our degeneracy loci. Finally, in Appendix \ref{appendixB} we give a Thom--Porteous type formula for the class of a degeneracy locus defined by partially decomposable three-forms.

\begin{ack}
	The authors wish to thank S.\ Druel for pointing out the proof of Lemma \ref{lemma:ratsing},  as well as 
	B.\ Fu and A.\ Garbagnati for useful references.
	The second author would like to thank Ch.\ Okonek for stimulating discussions and valuable advice during her stay in Zurich.\\
	This work has been carried out in the framework of the Labex Archim\`ede (ANR-11-LABX-0033) and 
	of the A*MIDEX project (ANR-11-IDEX-0001-02), funded by the ``Investissements d'Avenir" 
	French Government programme managed by the French National Research Agency. The second author was also partially supported by the Engineering and Physical Sciences Research Council programme grant ``Classification, Computation, and Construction: New Methods in Geometry'' (EP/N03189X/1). 
\end{ack}

\section{Geometric techniques for orbital degeneracy loci}
\label{geomtecnfor}

In this section we define, for an invariant subvariety $Y$ of a representation $V$ and a section $s$ of a vector bundle on a smooth variety $X$ having fiber $V$, the orbital degeneracy locus $D_Y(s)$. We show how a Kempf collapsing resolving the singularities of $Y$ can be used to construct a resolution of singularities of $D_Y(s)$. If the collapsing satisfies an additional crepancy condition, the canonical sheaf of such a resolution can be controlled in terms of the base variety $X$ and the vector bundle. Several examples are discussed.

\subsection{Orbital degeneracy loci} 

\label{gendegloc}
Let $G$ be an algebraic group acting on a variety $Z$. For any $G$-principal bundle 
$\cE$ over a manifold $X$, there is an associated bundle $\cE_Z$ over $X$ with fiber 
$Z$, defined as the quotient of $\cE\times Z$ by the equivalence relation 
$(eg,z)\simeq (e,gz)$ for any $g\in G$. In particular, if $V$ is a $G$-module, then
$\cE_V$ is a vector bundle over $X$, with fiber $V$. 

\begin{definition}
	Suppose that $V$ is a $G$-module and $Y$ a $G$-stable subvariety of $V$. Let $s$ be a global section of the vector bundle $\cE_V$. Then the $Y$-degeneracy locus of $s$, denoted by $D_Y(s)$, is 
	the scheme defined by the Cartesian diagram
	\[
	\xymatrix{\ar@{}[dr] |{\square}
		\cE_Y \ar[r] & \cE_V  \\
		D_Y(s) \ar[u] \ar@{^{(}->}[r] & X \ar[u]_-s }
	\]
	Its support is
	\[
	\{x\in X, \; s(x)\in \cE_Y \subset \cE_V\} = s^{-1}(\cE_Y).
	\]
\end{definition}
Under some mild assumptions, e.g.\ the generality of the choice of $s$, $D_Y(s)$ will have reduced structure and will be identified with its support.

If $E$ is a vector bundle of rank $e$ on $X$, the bundle of frames of $E$ is a 
$GL_e$-principal bundle $\cE$ on $X$, and $E=\cE_V$ for $V$ the natural representation
of $G=GL_e$. The only proper $G$-stable subvariety $Y$ of $V$ is the origin, and 
if $s$ is a global section of $E$, then $D_Y(s)$ is just the usual zero locus of $s$, which will be denoted by $\zero(s)$. 

If $F$ is another vector bundle of rank $f$ on $X$, the fiber product of the bundles of frames 
of $E$ and $F$ is a $GL_e\times GL_f$-principal bundle $\cH$ on $X$, and $\HOM(E,F)=\cH_V$ for $V$ 
the usual representation of $G=GL_e\times GL_f$ on the space  $V=M_{f,e}$ of matrices of size 
$f\times e$. The only closed $G$-stable subvarieties of $V$ are the varieties of matrices $Y_r$
of rank at most $r$, for $r\le \min(e,f)$.  
If $\varphi$ is a global section of $\Hom(E,F)$, then $D_{Y_r}(\varphi)$ is the usual $r$-th 
degeneracy locus of $\varphi$.

\subsection{Collapsing of vector bundles} 

A situation we will be interested in is when $Y\subset V$ is closed but singular, and can be 
desingularized by the total space of a homogeneous vector bundle; this is typically the case
of the varieties of matrices of bounded rank. 

Formally, suppose that $P$ is a parabolic subgroup 
of $G$, and that $W$ is a $P$-submodule of the $G$-module $V$. Then $G$ can be considered
as a $P$-principal bundle over the projective variety $G/P$, and we denote by $\cW$ and $\cV$ the 
vector bundles on $G/P$ associated to the $P$-modules $W$ and $V$. Obviously $\cW$ is a subbundle
of $\cV$. Moreover, since $V$ is a $G$-module, $\cV\simeq G/P\times V$ through the isomorphism 
induced by the map $(g,v)\mapsto (g,gv)$; in particular $\cV$ is (canonically) a trivial vector 
bundle on $G/P$ with fiber $V$. 
The second projection $\cV \rightarrow V$ restricts to a proper morphism $p_W$ mapping $\cW$ to its image $Y \subset V$; by construction $Y$ is a closed $G$-stable subvariety of $V$. 
This situation, illustrated in the commutative diagram \eqref{kempfCollapsing}, was originally described by Kempf \cite{Kempf76} and is sometimes referred to as a Kempf collapsing (of the vector bundle $\cW$).
\begin{equation}
	\label{kempfCollapsing}
	\xymatrix @C=2pc @R=0.4pc{
		\cV \ar[dd] & \rule{1pt}{0pt}\cW \ar@{_{(}->}[l] \ar[dr] \ar[dd]^-{p_W} \\
		& & G/P\\
		V & \rule{1pt}{0pt}Y \ar@{_{(}->}[l] 
	}
\end{equation}

\begin{theorem}[\cite{Kempf76}]
	\label{KempfInv}
	If $G$ is connected and $\cW$ is completely reducible, then $Y$ is normal and Cohen--Macaulay. If moreover $p_W$ is birational, it is a desingularization of $Y$ and $Y$ has rational singularities, i.e.\ ${p_W}_* \cO_{\cW}=\cO_Y$ and $\RRR^i {p_W}_* \cO_{\cW} = 0$ for any $i>0$.
\end{theorem}

This construction can be globalized as follows. From the $G$-principal bundle 
$\cE$ over $X$ we construct a variety $\cF_W$ as the quotient of $\cE\times G\times W$ by the 
equivalence relation $(e,h,w)\simeq (eg^{-1}, ghp^{-1},pw)$, for $g\in G$ and $p\in P$. 
The projection $p_{12}$ over the
first two factors induces a map $\cF_W\rightarrow \cE_{G/P}$ which makes $\cF_W$ a vector 
bundle over $\cE_{G/P}$, with fiber $W$. Moreover the map $(e,h,w)\mapsto (e,hw)$ induces 
a proper morphism $\cF_W\rightarrow \cE_V$, whose image is $\cE_Y$. This gives a relative 
version over $X$  of the morphism $\cW\rightarrow Y$.
In particular $\cF_W\rightarrow \cE_Y$
is birational when $p_W: \cW\rightarrow Y$ is birational. Note moreover that $\cF_V\simeq\theta^*\cE_V$,
if $\theta : \cE_{G/P}\rightarrow X$ is the projection map.
The inclusion $\cF_W \subset \cF_V$ induces the following short exact sequence of vector bundles on $\cE_{G/P}$:
\[
\xymatrix{
	0 \ar[r] &
	\cF_W \ar[r]&
	\cF_V \ar[r]^-\eta&
	Q_W \ar[r]&
	0
}.
\]

Consider now a global section $s$ of the vector bundle $\cE_V$ on $X$. Pulling it back to
$\cE_{G/P}$ and modding out by $\cF_W$, we get a global section $\tilde{s}:=\eta \circ \theta^*(s)$ of 
$Q_W$, whose zero locus maps to the $Y$-degeneracy locus of $s$:
$$\theta (\zero(\tilde{s})) = D_Y(s).$$
The relative version of \eqref{kempfCollapsing} is illustrated by the following commutative diagram:
\begin{equation}
	\label{relversion}
	\xymatrix @C=2pc @R=0.8pc{
		\cF_V \ar[dd] & \rule{1pt}{0pt}\cF_W \ar@{_{(}->}[l] \ar[dr]^-{p_{12}} \ar[dd] \\
		& & \cE_{G/P} \ar[dd]^-{\theta} & \ar@{_{(}->}[l] \ar[dd]^-{\theta'} \rule{1pt}{0pt}\zero(\tilde{s})\\
		\cE_V & \rule{1pt}{0pt}\cE_Y \ar@{_{(}->}[l] \ar[rd] \\
		&& X & \rule{1pt}{0pt}D_Y(s) \ar@{_{(}->}[l]
	}
\end{equation}

\begin{proposition}
	\label{propressing}
	Suppose that $\cE_V$ is globally generated and that $s$ is a general section. Then $\Sing D_Y(s)=D_{\Sing Y}(s)$. Moreover:
	\begin{itemize}[leftmargin=3.5ex]
		\item if $Y$ is normal (respectively, has rational singularities), then $D_Y(s)$ is normal (respectively, has rational singularities);
		\item if $p_W: \cW \rightarrow Y$ is birational, the restricted projection 
		$$\theta' : \zero(\tilde{s})\longrightarrow D_Y(s)\subset X$$
		is a resolution of singularities.
	\end{itemize}
\end{proposition}

\begin{proof} 
	The pullbacks of the global sections of $\cE_V$  generate the 
	quotient bundle $Q_W$ at every point of $\cE_{G/P}$, so the last part
	of the statement follows from the usual Bertini theorem whenever $p_W$ is birational.
	
	Consider the global degeneracy locus $D_Y(\cE)$, consisting of pairs $(x,s)$ with $s$ a section of $\cE_V$ 
	and $x$ a point of $X$ 
	such that $s(x)$ belongs to $\cE_Y$. Since $\cE_V$ is generated by global sections, $D_Y(\cE)$ is a locally 
	trivial fiber bundle over $X$, with fiber the product of $Y$ by an affine space. In particular $D_Y(\cE)$
	is singular exactly when $Y$ is singular, and its singular locus is 
	$D_{\Sing Y}(\cE)$. Bertini's theorem therefore implies our first claim.
	
	Finally, let $Y$ be normal (respectively, with rational singularities). Since the loci $D_Y(s)$ 
	are the fibers of the projection from $D_Y(\cE)$ to $\HHH^0(X,\cE_V)$, the normality (respectively, the rational singularities) of $D_Y(s)$ for $s$ general will follow from the next lemma, certainly well-known to experts. \qedhere
\end{proof}

\begin{lemma}
	\label{lemma:ratsing}
	Let $f: X\rightarrow B$ be a surjective morphism between irreducible varieties, and suppose that 
	$X$ has rational singularities. Then the general fiber of $f$ also has rational singularities. 
\end{lemma}

\begin{proof}
	Let $p: Y\rightarrow X$ be a resolution of singularities; $X$ has rational singularities if and only if  
	$p_*\cO_Y=\cO_X$ and  $\RRR^ip_*\cO_Y=0$ for $i>0$. Let $i_b: X_b\hookrightarrow X$ be the inclusion of a general 
	fiber of $f$, and $j_b: Y_b\hookrightarrow Y$  the inclusion of the corresponding fiber of $f\circ p$. 
	The restriction $p_b: Y_b\rightarrow X_b$ is a resolution of singularities. Applying the base change 
	statement \cite[Proposition 3.2]{Ou14}, we get
	$$\RRR^ip_{b*}\cO_{Y_b}=\RRR^ip_{b*}j_b^*\cO_Y=i_b^*\RRR^ip_*\cO_Y=0$$
	for $i>0$, and similarly $p_{b*}\cO_{Y_b}=i_b^*\cO_X=\cO_{X_b}$. Therefore $X_b$ has rational singularities. 
\end{proof}

\subsection{Parabolic orbits} 

An interesting source of orbital degeneracy loci is provided by $G$-modules
with finitely many orbits. Most of them come from $\theta$-groups \cite{Kac80},
which can be defined from gradings of semisimple Lie algebras. 

Let us restrict to $\ZZ$-gradings of simple Lie algebras. Suppose $\fg=\oplus_k \fg_k$
is such a grading; then $\fg_0$ is a Lie subalgebra, and each $\fg_k$ is a $\fg_0$-module. 
An example of $\ZZ$-grading is the one associated to a simple root $\alpha_i$, in the following way: given a root space decomposition
\[
\fg=\fh\oplus\bigoplus_{\alpha\in\Phi}\fg_\alpha,
\]
suppose that a set $\{\alpha_i\}$ of simple roots has been chosen. Consider the linear form $\ell$ on the root lattice such that 
$\ell(\alpha_i)=1$ and $\ell(\alpha_j)=0$ for $j\ne i$. Then 
\[
\fg_k=\bigoplus_{\ell(\alpha)=k}\fg_\alpha\oplus\delta_{k,0}\fh
\]
is a $\ZZ$-grading of $\fg$; moreover, $\fg_1$ is an irreducible $\fg_0$-module.

As it turns out, any $\ZZ$-grading of $\fg$ such that $\fg_1$ is irreducible is isomorphic to a grading associated to a simple root $\alpha_i$. In such a case, the semisimple part of $\fg_0$ has a Dynkin diagram deduced from that of $\fg$ just
by suppressing the node corresponding to $\alpha_i$. Moreover, $\alpha_i$ is
the lowest weight of $\fg_1$, so this irreducible $\fg_0$-module is easy to identify. 
Let $G_0$ be the subgroup of $G=\Aut(\fg)$ with Lie algebra $\fg_0$. By \cite[Lemma 1.3]{Kac80}, there 
are only finitely many $G_0$-orbits in $\fg_1$.

\begin{definition}
	A {\it parabolic orbit} is a $G_0$-orbit in $\fg_1$, obtained from some $\ZZ$-grading 
	of some simple Lie algebra $\fg$ associated to a simple root $\alpha_i$. 
\end{definition}

The terminology comes from the fact that, if $P_i$ is 
the maximal parabolic subgroup of $G$ defined by $\alpha_i$, then the cotangent 
bundle  to the homogeneous variety $G/P_i$ is the homogeneous vector bundle defined
by the $P_i$-module $\oplus_{k\ge 1} \fg_k$. 

\begin{fact}
	The singularities of a parabolic orbit closure can be resolved by a Kempf collapsing. 
\end{fact}

This should be taken with a caveat. In fact, the claim can be checked by hand for the classical 
types. The exceptional types were treated case by case in \cite{KW12,KW13}, except $E_8$, whose 
parabolic orbits remain a bit mysterious. 

\begin{examples}\rule{1pt}{0pt}
	\label{examplesParabolic}
	\begin{itemize}[leftmargin=3.5ex]
		\item[A.]
		Consider $\fg=\fsl_{e+f}$ and the $\ZZ$-grading defined by the simple 
		root $\alpha_e$. Then the action of $G_0$ on $\fg_1$ is essentially the action of $GL_e\times GL_f$
		on the space of matrices $M_{f,e}$. In particular the parabolic orbits for this case
		are just the spaces of matrices of a given rank.
		\item[B.]
		Consider $\fg=\mathfrak{sp}_{2e}$ and the $\ZZ$-grading defined by the simple root $\alpha_e$. Then the action of 
		$G_0$ on $\fg_1$ is essentially the action of $GL_e$
		on the space $\Sym_e$ of symmetric matrices of size $e$. In particular the parabolic orbits for 
		this case are just the spaces of symmetric matrices of a given rank. Similarly, from the orthogonal Lie algebras we would get the spaces of skew symmetric matrices of a given rank.
		\item[C.]
		Consider $\fg=\fe_6$ and the $\ZZ$-grading defined by the simple root $\alpha_2$, corresponding to the 
		adjoint representation. Then the action of 
		$G_0$ on $\fg_1$ is essentially the action of $GL_6$ on $\wedge^3\CC^6$. The orbit decomposition 
		in this case is very simple, since the orbit closures form a string \cite{Donagi77}
		\begin{equation}
			\label{donagiChain}
			0=Y_0\subset \overline{Y_1}\subset \overline{Y_2}\subset \overline{Y_3}\subset \overline{Y_4}= \wedge^3\CC^6.
		\end{equation}
		Here $Y_1$ is the space of non-zero fully decomposable tensors $v_1\wedge v_2\wedge v_3$
		(a cone over the Grassmannian $\Gr(3,6)$); $\overline{Y_3}$ is a degree four hypersurface, which can be defined
		as the closure of the union of the tangent spaces to $Y_1$. The closure of $Y_2$ is the $15$-dimensional variety of partially decomposable 
		tensors $v\wedge\omega$, where $v\in \CC^6$ and $\omega \in\wedge^2\CC^6$; it is singular along $\overline{Y_1}$, hence in codimension $5$. We will focus on this special variety in Section \ref{partDecForms}, where we will construct many varieties with trivial canonical bundle as $D_{\overline Y_2}$-degeneracy loci.
	\end{itemize}
\end{examples}

\subsection{The canonical sheaf}

We will be interested in the canonical sheaf of orbital
degeneracy loci. The following key result will allow us to get some control on this sheaf:

\begin{proposition}
	Suppose that $Y$ has rational singularities and admits a birational Kempf collapsing $p_W: \cW\rightarrow Y$ such that 
	\begin{equation}
		\label{condCrepancy}
		K_{G/P}=det(\cW).
	\end{equation} 
	Then the canonical sheaf of $Y$ is trivial. Moreover, if  ${\bar Y}\subset\PP (V)$ denotes the projectivization of the cone $Y$, then 
	the induced  resolution of singularities ${\bar p}_W: \PP(\cW)\rightarrow {\bar Y}$ is crepant. 
\end{proposition}

\begin{proof}
	Condition \ref{condCrepancy} clearly implies that the canonical sheaf of the total space $\cW$ is trivial. 
	Since $Y$ has rational singularities, $K_Y=p_{W*}K_\cW$, so $K_Y$ is also trivial. 
	
	If $w$ denotes the rank of the vector bundle $\cW$, the canonical bundle of its projectivization is 
	$$K_{\PP(\cW)}=\cO_{\PP(\cW)}(-w)={\bar p}_W^*\cO_{{\bar Y}}(-w).$$
	Since ${\bar Y}$ also has rational singularities, we deduce that its canonical sheaf is $K_{{\bar Y}}={\bar p}_{W*}
	K_{\PP(\cW)}=\cO_{{\bar Y}}(-w)$, and therefore $K_{\PP(\cW)}={\bar p}_W^*K_{{\bar Y}}$. 
\end{proof}

In the relative setting, this has the following crucial consequence. 

\begin{proposition}
	\label{Gorcansing}
	Suppose that the Kempf collapsing $p_W: \cW\rightarrow Y$ satisfies condition \eqref{condCrepancy}.
	If $\cE_V$ is globally generated and $s$ is a general section, then the canonical sheaf of $\zero(\tilde{s})$  
	is the restriction of the pull-back of some line bundle $L$ on $X$. If moreover $p_W$ is birational and $Y$ has rational singularities, then $D_Y(s)$ is Gorenstein, has canonical singularities and its canonical bundle is the restriction of $L$.
	
	\begin{proof}
		Recall that $\zero(\tilde{s})$ is the zero locus of a section 
		of $Q_W=\theta^*\cE_V/\cF_W$ on $\cE_{G/P}$, which is in general transverse to the zero section. 
		Therefore, its canonical sheaf can be computed as the restriction to $\zero(\tilde{s})$ of 
		$$K_{\cE_{G/P}}\otimes \det(Q_W) = K_{\cE_{G/P}/X}\otimes \det(\cF_W)^*\otimes\theta^*(K_X\otimes\det(\cE_V)).$$
		The restriction to each fiber of $\theta$ (a copy of $G/P$) of the line bundle 
		$K_{\cE_{G/P}/X}\otimes \det(\cF_W)^*$ is isomorphic to $K_{G/P}\otimes \det(\cW)^*$, hence trivial
		under our hypothesis. Thus $K_{\cE_{G/P}/X}\otimes \det(\cF_W)^*$ must be the 
		pullback of some line bundle from $X$, and the same conclusion holds for $K_{\cE_{G/P}}\otimes \det(Q_W)$. 
		So there is a line bundle $L$ on $X$ such that 
		$$K_{\zero(\tilde{s})}=(\theta^*L)|_{\zero(\tilde{s})}.$$
		If $p_W$ is birational and $Y$ has rational singularities, by Proposition \ref{propressing} $D_Y(s)$ has rational singularities and its canonical sheaf is
		\[K_{D_Y(s)}=\theta_*K_{\zero(\tilde{s})}=L|_{D_Y(s)}. \]
		Then $D_Y(s)$ is Gorenstein and has canonical singularities (see e.g.\ \cite[Corollary 11.13]{Kollar97}).
	\end{proof}
\end{proposition}

\begin{remark}
	By Proposition \ref{Gorcansing}, even if $p_W$ is not birational we can still conclude that $K_{\zero(\tilde{s})}=
	(\theta^*L)|_{\zero(\tilde{s})}$.
	For instance, for $Y$ given by the closure of particular Richardson orbits (see Section \ref{nilpotOrbits}), $p_W$ has degree two. In this situation we can still consider diagram \eqref{relversion}; $\zero(\tilde{s})$ is a variety with trivial canonical bundle, endowed with an interesting birational involution given by the degree two map $\theta'$.
\end{remark}

\subsection{First examples} 

\begin{example} \label{matriceskempf} Let $V_e, V_f$ be vector spaces of dimensions $e,f$ respectively. 
	Fix an integer $r<\min(e,f)$. Denote by $\cU$ the tautological vector bundle on the Grassmannian $\Gr(r,V_f)$, 
	and by $\cW$ the vector bundle $\HOM(V_e,\cU)$. The total space of this bundle is a desingularization 
	of the variety $Y_r$ of morphisms of rank at most $r$ inside  $\Hom(V_e,V_f)$. Moreover $\det(\cW)=\det(\cU)^e$,
	while $K_{\Gr(r,V_f)}=\det(\cU)^f$, so that condition \eqref{condCrepancy} is fulfilled if and only if $e=f$, and then for any $r$. 
	
	Note that, in a dual way, we could also have chosen $\cW=\HOM(V_e/{\cal T},V_f)$, with ${\cal T}$ the  tautological vector bundle 
	on the Grassmannian $\Gr(e-r,V_e)$. This yields another desingularization 
	of the variety $Y_r$ satisfying condition \eqref{condCrepancy}, related to the previous one by a Mukai flop. 
	
	Another, more symmetric choice would be the bundle $\cW=\HOM(V_e/{\cal T},\cU)$ on $\Gr(e-r,V_e)\times \Gr(r,V_f)$. 
	But then condition \eqref{condCrepancy} is NOT satisfied.
\end{example}

\begin{remark}
	This example explains why it is possible to construct varieties with trivial canonical bundle as classical degeneracy loci of morphisms between vector bundles {\it of the same rank}. In fact, a few Calabi--Yau degeneracy loci of (possibly symmetric or skew-symmetric) morphisms between vector bundles have already been described. Tonoli constructed Pfaffian Calabi--Yau threefolds in $\PP^6$ \cite{Ton04}; his construction was later generalized by Kanazawa \cite{Kan12}, who replaced the ambient space by weighted projective spaces. Determinantal Calabi--Yau threefolds have been also studied from a different perspective in \cite{GP01} (see also \cite{Ber09}), and further examples have been explicitly described in \cite{Kap11}.
	
	Pfaffian orbit closures are examples of subvarieties $Y$ such that the canonical bundle of a $Y$-degeneracy locus can be controlled even if no resolution of $Y$ satisfying condition \eqref{condCrepancy} is known. This behavior, which is typical of Gorenstein orbit closures or subvarieties, is explained and investigated in \cite{ODL2}.
\end{remark}

\begin{example} 
	\label{krdecomposable} Let again $\cU, \cQ$ denote the tautological and quotient vector bundles on a
	Grassmannian $\Gr(r,V_d)$. Let $k \leq r \leq d$, $k+\ell \leq d$ and let $\cW=\wedge^k\cU\wedge\wedge^\ell V_d$, a subbundle of the trivial
	bundle $\wedge^{k+\ell} V_d$.  Then the total space of $\cW$ maps to
	\[
	Y_{k,r}:=\left\{\begin{array}{cc}\omega \in \wedge^{k+\ell} V_d, \; \omega=\sum \alpha_i \wedge \beta_i \mbox{ such that } \beta_i \in \wedge^\ell V_d\\ \mbox{ and } \alpha_i \in \wedge^k U \mbox{ for some } U \subset V_d \mbox{ of dimension } r\end{array} \right\},
	\]
	the variety of $(k,r)$-decomposable forms inside $\wedge^{k+\ell} V_d$. Beware that this collapsing will in general be a desingularization, but not always.
	
	Note that $\cW$ has a natural filtration whose quotients are the bundles $\wedge^{k+i}\cU\otimes\wedge^{\ell-i} \cQ$,
	for $i\leq\min(r-k,\ell)$. We deduce that $\det(\cW)=\det(\cU)^N$ for 
	\[
	N=\sum_{i=0}^{\min(r-k,\ell)}\frac{(r-1)!(d-r-1)!}{(k+i)!(r-k-i)!(\ell-i)!(d-r-\ell+i)!}\left( 
	(k+i)d-(k+\ell)r \right);
	\]
	hence, condition \eqref{condCrepancy} is satisfied when $N=d$, a diophantine equation with infinitely many solutions. 
	
	A simple solution is $\ell=0$, $k=3$, $d=10$, $r=6$. The quotient bundle $Q_W=\theta^*\wedge^3E/\wedge^3\cU$ has rank $100$, so $\zero(\tilde{s})$ has dimension 
	and canonical sheaf
	$$\dim \zero(\tilde{s})=\dim X-76, \qquad K_{\zero(\tilde{s})}=\theta^*(K_X\otimes (\det E)^{30})|_{\zero(\tilde{s})}.$$
	So, in order to construct for example a fourfold with trivial canonical class, we would need a Fano variety 
	$X$ of dimension $80$, and a rank $10$ vector bundle $E$ on $X$ such that $\wedge^3E$ is globally 
	generated and  $K_X=(\det E)^{-30}$. 
	
	Another simple solution is $\ell=2$, $k=1$, $d=10$, $r=4$, which corresponds to the hyperk\"ahler variety described by Debarre--Voisin in \cite{DV10}; $\cW$ is the kernel bundle of the map $\wedge^3 V_{10} \to \wedge^3\cQ$ over $\Gr(4,V_{10})$. Therefore, $Q_W=\wedge^3\cQ$. In this case $\cW$ cannot be a desingularization of $Y_{1,4}$ for dimensional reasons: indeed
	$$\dim \zero(\tilde{s})=\dim X+4, \qquad K_{\zero(\tilde{s})}=\theta^*(K_X\otimes (\det E)^{6})|_{\zero(\tilde{s})}.$$
	In order to obtain a fourfold, $X$ has to be a point, and in this way one recovers the hyperk\"ahler family constructed by Debarre and Voisin.
	
	Finally, the solution $\ell=2$, $k=1$, $d=6$, $r=1$ gives a desingularization of the variety $\overline{Y_2}$ of partially decomposable forms in $\wedge^3 \CC^6$ appearing in \eqref{donagiChain}, as we will see in the next section more in detail.

\end{example}

\begin{example} More generally, choose a partition $\lambda$ with at most $r$ non-zero parts. Let us denote by $S_{\lambda}$ the Schur functor associated to $\lambda$, i.e., for instance, $S_{(1^k)}V = \wedge^k V$. Consider  on the Grassmannian $G=\Gr(r,V_d)$ the vector bundle $\cW=S_{\lambda}\cU$, a subbundle of the trivial
	bundle $S_{\lambda}V_d$. The total space of $\cW$ is a desingularization of the rank $r$ variety $Y_r$  inside $S_{\lambda}V_d$ \cite{Porras96}, which has rational singularities by Theorem \ref{KempfInv}.
	Let $r_\lambda$ be the rank of $\cW$, and define $d_\lambda$ by the identity $\det \cW= \cO_G(-d_\lambda)$. These
	integers are given by 
	$$d_\lambda=\frac{|\lambda|r_\lambda}{r}, \qquad r_\lambda=\frac{\prod_{x\in D(\lambda)}(r+c(x))}{h(\lambda)},$$ 
	where $|\lambda|$ denotes the size of $\lambda$ (the sum of its parts), $D(\lambda)$ is the diagram of $\lambda$ (with $\lambda_i$ boxes on the $i$-th row), where a box $x=(i,j)$ in this diagram has content $c(x)=j-i$, and 
	$h(\lambda)$ is the product of the hook lengths. 
	
	Condition \eqref{condCrepancy} is fulfilled exactly when $d=d_\lambda$. Note that in general the singular locus of $Y_r$
	is $Y_{r-1}$, and has large codimension in $Y_r$. 
	
	A concrete example is the following: let us consider the partition $\lambda=(2,1)$. Then we need $d=r^2-1$. 
	So let $E$ be a vector bundle of rank $d$ on $X$, such that $S_{21}E$ is generated by global sections. 
	If $s$ is a general section, $D_{Y_r}(s)$ has dimension 
	$$\dim D_{Y_r}(s) = \dim X +r(d-r)+\frac{r(r^2-1)}{3}-\frac{d(d^2-1)}{3}$$
	and its canonical sheaf is given, with the same notation as before, by 
	$$K_{D_{Y_r}(s)}=(K_X\otimes (\det E)^{d^2-1-r})|_{D_{Y_r}(s)}.$$
\end{example}

\begin{remark} Example \ref{matriceskempf}  shows that, in general, 
	\begin{enumerate}
		\item there are potentially several non-equivalent ways to desingularize a
		$G$-variety by total spaces of homogeneous vector bundles; 
		\item only some of them, if any, will satisfy condition \eqref{condCrepancy}. 
	\end{enumerate}
	It would be important to classify birational collapsings
	of vector bundles satisfying \eqref{condCrepancy}. Several new examples are exhibited in \cite{ODL2}.
\end{remark}

\section{Partially decomposable forms}
\label{partDecForms}

In this section we consider degeneracy loci associated to the orbit of partially decomposable three-forms in six variables. We present some general constructions and produce several examples of threefolds and fourfolds with trivial canonical bundle, all of which turn out to be Calabi--Yau varieties.

\subsection{General setting}

Let $V_6$ be a six dimensional complex vector space. As mentioned in Example \ref{examplesParabolic} C., the action of $GL(V_6)$ on the space of skew-symmetric three-forms $\wedge^3V_6$ has only five orbits, whose closures form the chain \eqref{donagiChain}. The orbit closure we will focus on is $Y=\overline{Y_2}$. Its singular locus is $\overline{Y_1}$ and there are several natural ways to resolve its (rational) singularities.

Let $\cO(-1)$ denote the tautological line bundle on  $\PP(V_6)$, and let $\cW_1=\cO(-1)\wedge (\wedge^2V_6)$, a subbundle of
the trivial vector bundle $\wedge^3V_6$. Then the total space of $\cW_1$ collapses to $Y$ and provides a first 
desingularization. Since $\cW_1=\cO(-1)\otimes\wedge^2\cQ$,
with $\cQ$ the tautological quotient bundle on $\PP(V_6)$, we compute that $\det \cW_1=\cO(-6)$, so that 
condition \eqref{condCrepancy} is satisfied. Note that this desingularization corresponds to the desingularization of the variety of $(1,1)$-decomposable forms inside $\wedge^{3}V_6$, see Example \ref{krdecomposable}.

In a dual way ($\wedge^3V_6$ is in fact self-dual), we could also have chosen $\cW_2=\wedge^3\cU$, with $\cU$ the  
tautological vector bundle on the Grassmannian $\Gr(5,V_6)=\PP(V_6^*)$. This yields another  desingularization 
of the variety $Y$, again satisfying condition \eqref{condCrepancy}, and related to the previous one by a flop. 

A more symmetric choice would be the bundle $\cW_3=\cL\wedge\wedge^2\cU$ on the flag variety $F(1,5,V_6)$, 
where $\cL \subset \cU$ denote the rank one and rank five tautological bundles. This desingularization dominates the previous ones (as shown in the following diagram), but condition \eqref{condCrepancy} is NOT satisfied. 
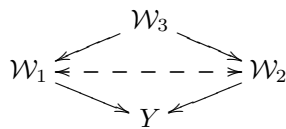
\[
\xymatrix @C=2pc @R=0.4pc{
	& \rule{1pt}{0pt}\cW_3 \ar[dl] \ar[dr]  \\
	\cW_1\ar[dr] \ar@{-->}[rr]& & \cW_2 \ar[dl] \ar@{-->}[ll] \\
	& \rule{1pt}{0pt}Y 
}
\]
\medskip In the relative setting, we consider a vector bundle $E$ of rank $6$ on a variety $X$. Following the notation of 
Section \ref{gendegloc}, we consider the $GL_6$-principal bundle $\cE$ of frames of $E$; then 
$\cE_{\wedge^3 V_6}\cong \wedge^3 E$. If $s$ is a section of this bundle over $X$, its $Y$-degeneracy locus is
\[D_Y(s)=\{x\in X, \; s(x) \mbox{ is partially decomposable in } 
\wedge^3E_x\}.\] 
For $\wedge^3 E$ generated by global sections, and $s$ general, $D_Y(s)$ will be of codimension five in $X$, and singular 
exactly at the points where $s(x)$ is completely decomposable, a sublocus of codimension five in 
$D_Y(s)$ (see Proposition \ref{propressing}). Moreover its singularities will be resolved by the zero locus
$\zero(\tilde{s})$ inside $\cE_{G/P}\cong \PP(E)$, where $\tilde{s}$ is the induced section of $Q_W$. If we denote by 
$\cO_{\PP(E)}(-1)$ (respectively $\cQ_{\PP(E)}$) the tautological subbundle (quotient bundle) over $\PP(E)$, we have
\[
Q_W=\theta^*\cE_{\wedge^3 V_6}/\cF_W\cong \theta^*(\wedge^3 E) / (\cO_{\PP(E)}(-1)\otimes\wedge^2 \cQ_{\PP(E)})\cong \wedge^3 \cQ_{\PP(E)}.
\]
We compute the canonical bundle of $\zero(\tilde{s})$ from the adjunction formula:
\[
K_{\zero(\tilde{s})}=(K_{\PP(E)}\otimes \det Q_W)|_{\zero(\tilde{s})}= 
\theta^*(K_X \otimes (\det E)^5)|_{\zero(\tilde{s})}.
\] 

The statement we will use in the sequel is the following. 

\begin{proposition}
	For $d\le 4$, let $X$ be a projective variety of dimension $d+5$. Let $E$ be a rank six vector bundle  on $X$, such 
	that $K_X=(\det E)^{-5}$ and $\wedge^3E$ is generated by global sections. Let $s$ be a general section. 
	Then $D_Y(s)$, the locus of points where the section $s$ becomes partially decomposable, 
	is either empty or smooth of dimension $d$, with trivial canonical bundle. 
\end{proposition}

Our problem in the sequel will therefore mainly be the following:

\medskip\noindent {\sc Problem}. Find projective varieties $X$, of dimension eight or nine, endowed with a 
non-trivial vector bundle $E$ of rank six such that $\wedge^3E$ is globally generated, and 
\begin{equation}
	\label{condwedge3}
	K_X=(\det E)^{-5}.
\end{equation}

\subsection{Constructions}

The assumptions on the variety $X$ and the vector bundle $E$ are somehow restrictive. On the one hand, $\det E$ must be semiample and non-trivial, and therefore $K_X^{-1}$ too. On the other hand, the index of $X$ has to be a multiple of $5$; by the Kobayashi--Ochiai inequality \cite{KO73} (see also \cite{IP99}),
\begin{equation}
	\label{KOin}
	\ind (X) \leq \dim X + 1.
\end{equation}
If $\ind (X) \neq 5$, it has to be $10$ and then $X=\PP^9$. We are not aware
of any suitable rank six vector bundle on $\PP^9$ other than $\cO(2)\oplus 5\cO$ or $2\cO(1)\oplus 4\cO$. We will therefore restrict our search to varieties $X$ with index $5$.  

\medskip
If $K_X=L^{-5}$ for some (non-trivial) globally generated line bundle $L$,
a naive possibility would be to 
consider $E=L\oplus 5\cO_X$. We will rule out this case from our study because of the following:

\begin{proposition}
	\label{lemmaE=5O+O1}
	If $E=L\oplus 5\cO_X$, with $L$ a globally generated line bundle, then the degeneracy locus $D_Y(s)$ arising from 
	a general section $s$ of $\wedge^3 E$ is the zero locus of a general section of $5L$.
\end{proposition}

\begin{proof}
	Let us write $E=V_5\otimes \cO_X\oplus L$, for a five dimensional vector space $V_5$. Then
	\[
	\wedge^3 E=\wedge^3 V_5\otimes \cO_X\oplus \wedge^2V_5\otimes L,
	\]
	so that a section $s\in \HHH^0(X, \wedge^3 E)$ can be decomposed as $s=\sigma+s'$, where $\sigma \in \wedge^3 V_5$ and $s'\in  \wedge^2 V_5\otimes \HHH^0(X, L)$. In general $\sigma$, considered as a 
	two-form by the isomorphism $\wedge^3 V_5\simeq \wedge^2 V_5^*$, will have rank four; dually, this exactly means that it can be decomposed as $\sigma=v_0\wedge\omega_0$, where 
	$v_0\in V_5$ and $\omega_0\in \wedge^2 V_5$. The vector $v_0$ generates the kernel of $\sigma$, in particular it is uniquely 
	defined up to scalar. The two-form $\omega_0$ is unique up to a wedge product of $v_0$ by another vector. 
	
	At a point $z\in X$, let $\lambda$ be a generator of the fiber $L_{z}$; then
	\[
	s(z)=\sigma+s'(z)=\sigma+\omega\otimes \lambda
	\]
	where $\omega\in \wedge^2 V_5$. It is partially decomposable if we can factor it out as $s(z)=(v+c\lambda) \wedge (\phi+\theta\otimes\lambda)$, where $v, \theta \in V_5$, $c\in \mathbb{C}$, $\phi\in \wedge^2 V_5$. This is equivalent 
	to the two identities
	\[
	\sigma=v\wedge \phi, \qquad \omega=v\wedge \theta-c\phi.
	\]
	The first equation implies that $v=tv_0$ for some $t\ne 0$, and $\phi=t^{-1}\omega_0+v_0\wedge w$ for some 
	$w\in V_5$. The second equation can then be solved if and only if $\omega$ belongs to the codimension five subspace $U$ of 
	$\wedge^2 V_5$ spanned by $\omega_0$ and $v_0\wedge V_5$. We conclude that our degeneracy locus $D_Y(s)$ can be  
	defined by the condition that the section of $( \wedge^2 V_5 / U)\otimes L$ induced by $s'$ vanishes, and our
	claim follows. 
	\qedhere
	
\end{proof}
Our problem can therefore be approached as follows:
\begin{enumerate}
	\item Find Fano varieties $X$ of dimension eight or nine and index $5$, that is $K_X=L^{-5}$ for some 
	ample line bundle $L$. 
	\item Find vector bundles $E$ of rank six on those $X$, not of the form $L\oplus 5\cO_X$, such that $\det E=L$. 
	Moreover $\wedge^3E$ must be generated by global sections. 
\end{enumerate}

Fano varieties of dimension eight and index five are close to {\it Mukai varieties}, which are
Fano varieties of dimension $n$ and index $n-2$. Mukai varieties are (almost) classified in \cite{Mu89} (see also \cite{CLM98}). Roughly speaking, they consist in: 1) complete intersections; 2)
branched covers; 3) sections of rational homogeneous varieties;  4) blow-ups; 5) projective bundles, including 
products. 

This suggests that we look for varieties of similar types. For types 1), 2) and 4), unfortunately we do not have suitable vector bundles, so we will restrict our study to two types of varieties: subvarieties of homogeneous 
spaces, and projective or Grassmannian fibrations. The possibility of constructing Calabi--Yau varieties in homogeneous spaces has already been considered, e.g.\ by H\"ubsch \cite{Hubsch92}. Let us briefly discuss the latter type. 

\subsubsection{Grassmann bundles}
\label{grahomvar}

Consider  a Grassmann bundle $\pi:X=\Grb(k,F)\to Z$, where $F$ is a bundle on $Z$. In this situation, if ${\cal U}_{X/Z}$ denotes the tautological subbundle of rank $k$ on $\Grb(k,F)$, we have:
\[
K_X=\det({\cal U}_{X/Z})^{\rank(F)}\otimes \pi^*(K_Z\otimes \det(F^*)^k).
\]

As we want $K_X$ to be divisible by $5$, we have to impose some conditions on $Z$ and $F$. For example, we can ask for the following two properties:
\begin{equation}
	\label{condgrassbundle1}
	K_Z\otimes \det(F^*)^k={\cal O}_Z, \qquad \rank(F)=5
\end{equation}
This implies that $1\leq k \leq 4$.
\begin{itemize}[leftmargin=2.5ex]
	\item $k=1$. 
	But then there is no obvious choice for $E$, apart from $E={\cal U}_{X/Z}^*\oplus 5{\cal O}_X$
	that we have excluded.
	
	\item $k=2$. The variety $Z$ has dimension $2$ or $3$, i.e.\ it is a del Pezzo surface or a Fano threefold. Moreover, \eqref{condgrassbundle1} implies that the index of $Z$ is divisible by $2$. The only del Pezzo surface with this property is $\PP^1\times \PP^1$. If $Z$ has dimension three, it must be a del Pezzo threefold (recall that del Pezzo manifolds are Fano manifolds of dimension $n$ whose index is divisible by $n-1$; they were classified by Fujita, see \cite{IP99} and references therein). 
	A natural choice for $E$ is $E={\cal U}_{X/Z}^*\oplus 4{\cal O}_X$.
	
	\item $k=3$.  The variety $Z$ has dimension $2$ or $3$ as before, but now the index of $Z$ is divisible by $3$. As for del Pezzo surfaces, the only possibility is $\PP^2$. If $\dim(Z)=3$, $Z$ must be a quadric in $\PP^4$. A natural choice for $E$ is $E={\cal U}_{X/Z}^*\oplus 3{\cal O}_X$.
	
	\item $k=4$. The variety $Z$ has dimension $4$ or $5$, and the index of $Z$ must be divisible by $4$. If $Z$ has dimension $4$, it must be a quadric in $\PP^5$. If $\dim(Z)=5$, it must be a del Pezzo fivefold.
	A natural choice for $E$ is $E={\cal U}_{X/Z}^*\oplus 2{\cal O}_X$.
\end{itemize}

We can also replace conditions \eqref{condgrassbundle1} by
\begin{equation}
	\label{condgrassbundle3}
	K_Z\otimes \det(F^*)^k=L^5, \qquad \rank(F)=5,
\end{equation}
for some line bundle $L$ on $Z$. Then we need to choose $E$ such that $\det(E)=\det({\cal U}^*_{X/Z})\otimes \pi^*L^*$. For example we can consider $k=1$ and 
$Z=\PP^5$; we can then set $F=\cO_{\PP^5}(-1)\oplus 4\cO_{\PP^5}$ or $F=\cQ_{\PP^5}^*$, and $L=\cO_{\PP^5}(-1)$, for which 
$E={\cal U}^*_{X/Z}\oplus \pi^* L^*\oplus 4\cO_X$ produces a $Y$-degeneracy locus with trivial canonical bundle inside $\Grb(k,F)$.

We also notice that if $F$ is a trivial bundle we get products of the form $Z\times \Gr(k,5)$. As we want $E$ to depend on $Z$ to avoid trivial cases, we suppose that $Z$ is Fano, and therefore of index $5$. Then, if $k=1$, we obtain $\PP^4 \times \PP^4$ with different possibilities for $E$ (see Section \ref{subsec:threefolds}), or $\mathbb{Q}^5 \times \PP^4$, where $\mathbb{Q}^n$ is the quadric of dimension $n$ (see Section \ref{fourfoldssect}). The 
case $k=2$, is excluded by the Kobayashi--Ochiai inequality \eqref{KOin}.

\medskip
Many other choices are of course possible. Instead of considering $X=\Grb(k,F)$, we could take $X$ as the zero locus of a section of a vector bundle on a suitable Grasmann bundle. A systematic study of these cases, however, falls outside the scope of this paper.

\subsubsection{Twisted degeneracy loci}
\label{twisteddeg}

Consider a vector bundle $E$ of rank $6$ and a line bundle $L$ on $X$. Then, taking a section 
$s\in \wedge^3E \otimes L$, one can consider the \emph{twisted} $Y$-degeneracy locus $D_Y(s)\subset X$ consisting of points $x\in X$ such that $s(x)$ is in the \emph{twisted} fibration 
$\cE_Y \otimes L\subset \wedge^3 E\otimes L$. 
In this new situation, the canonical bundle of the resolution 
$\zero(\tilde{s})\subset \PP(E)$ becomes
\begin{multline*}
	\left(\theta^*(K_X)\otimes K_{\PP(E)/X} \otimes \det(\wedge^3 E \otimes L) \otimes \det(\wedge^2\cQ_{\PP(E)}(-1) \otimes L)^*\right)|_{\zero(\tilde{s})} = \\
	= (\theta^*(K_X \otimes \det(E)^{5} \otimes L^{10}))|_{\zero(\tilde{s})};
\end{multline*}
hence, ${\zero(\tilde{s})}$ has trivial canonical bundle if
\begin{equation}
	\label{eqtwisteddeg}
	K_X \otimes \det(E)^{5} \otimes L^{10} = \cO_X.
\end{equation}

It is easy to see that this condition is coherent with condition \eqref{condwedge3} when $L=L'^{3}$, which implies 
$\wedge^3 E \otimes L=\wedge^3 (E\otimes L')$. 
As we require that the bundle 
$\wedge^3 E \otimes L$ is globally generated, we have a restriction on the choice of $L$ and $E$. On the one hand we can choose the two of them to be globally generated; say 
$L={\cal O}(1)$, and $\det(E)={\cal O}(1)$, with ${\cal O}(1)$ ample and primitive. Then condition \eqref{eqtwisteddeg} becomes $K_X={\cal O}(-15)$ and the Kobayashi--Ochiai inequality implies that $\dim(X)\geq 14 >9$. On the other hand, let us assume $L={\cal O}(-1)$; then, for example, we can choose 
$E=4{\cal O}(1) \oplus 2{\cal O}_X$ and in this way 
\[
\wedge^3 E \otimes L=(4{\cal O}(3)\oplus 12{\cal O}(2)\oplus 4{\cal O}(1))\otimes {\cal O}(-1)=4{\cal O}(2)\oplus 12{\cal O}(1)\oplus 4{\cal O}_X
\]
is globally generated. Condition (\ref{eqtwisteddeg}) becomes $K_X={\cal O}(-10)$. 
The Kobayashi--Ochiai inequality implies again that $\dim(X)=9$, and $X=\PP^9$.

\subsubsection{Simple connectedness}
One natural question that arises when constructing (Calabi--Yau) varieties is whether they are simply connected. We are able to prove the simple connectedness of our orbital degeneracy loci in the case of partially decomposable forms when the base variety $X$ is a complete intersection.
\begin{proposition}
	\label{propsimpleconn}
	Let $X$ be a variety of dimension at least seven which is the zero locus of a general section of an ample line bundle $L$ over a variety $X'$. Suppose that there exists a vector bundle $E$ on $X'$ of rank six, such that $\wedge^3 E$ is globally generated, and $K_X=\det(E|_X)^{-5}$.  Consider the degeneracy locus $D_Y(s)\subset X$, where $s$ is a general section of $\wedge^3 E|_X$. 
	Then the desingularization $\zero(\tilde{s})$ of $D_Y(s)$ inside $\PP(E|_X)$
	is simply connected.
\end{proposition}

\begin{proof}
	We will prove the simple connectedness of $\zero(\tilde{s})$ when the section $s$ is the restriction of a general section $t$ of $\wedge^3 E$ over $X'$. Then, a deformation argument implies our assertion.   
	
	The idea of the proof is to use some generalizations of the Lefschetz hyperplane theorem to prove the vanishing of  relative homotopy groups (see for example \cite{SV86}). In particular, we want to apply \cite[Corollary 22]{Ok87}, which states that if $Z$ is the zero locus of a section of a globally generated $k$-ample vector bundle over $Z'$, then the relative homotopy groups $\pi_i(Z',Z)$ are trivial for $i\leq \dim(Z) - k$. Let us  recall the definition of $k$-ampleness (first introduced by Sommese in \cite{So78}): a line bundle $L$ on $Z'$ is $k$-ample if $L^r$ is globally generated for some $r>0$, and the fibers of the corresponding morphism $\phi: Z'\to \HHH^0(Z',L^r)$ have dimension at most $k$.
	
	In our situation, even though $L$ is ample (i.e.\ $0$-ample) over $X'$, this variety is non-necessarily simply connected. We will rather consider our orbital degeneracy locus as a subvariety of  another degeneracy locus, that will be (almost) Fano and therefore simply connected. Moreover, the fact that in higher dimensions degeneracy loci are singular will force us to work on their desingularizations.
	
	Denote by $D_Y(t)\subset X'$ the degeneracy locus associated to the section $t \in \HHH^0(X',\wedge^3 E)$, and suppose $s=t|_X$. As $X\subset X'$ is the zero locus of a section of $L$, $D_Y(s)\subset D_Y(t)$ 
	is as well the zero locus of a section of $L|_{D_Y(t)}$.
	Similarly, when we pass to the respective desingularizations, we have that $\zero(\tilde{s})\subset \zero (\tilde{t})$ 
	is the zero locus of a section of $\theta^*(L)|_{\zero(\tilde{t})}$. The following diagram illustrates this situation:
	\begin{equation*}
		\xymatrix @C=2pc @R=0.4pc{
			\zero(\tilde{s}) \ar@{^{(}->}[r] \ar[dd] & \rule{1pt}{0pt} \zero(\tilde{t}) \ar@{^{(}->}[r] \ar[dd] & \rule{2pt}{0pt} \PP(E) \ar[dd]^-{\theta}\\
			& & \\
			D_Y(s) \ar@{^{(}->}[r]  & \rule{1pt}{0pt}D_Y(t) \ar@{^{(}->}[r] & \rule{2pt}{0pt} X'
		}
	\end{equation*}
	
	In order to apply Okonek's result, we have to verify that $\theta^*(L)|_{\zero(\tilde{t})}$ is $k$-ample for a suitable $k$. The value of $k$ will depend on the dimension of the fibers of $\theta$, i.e.\ on the dimension of $X'$. 
	\begin{itemize}
		\item If $\dim(X')\leq 9$, then $\zero(\tilde{t})\cong D_Y(t)$ and $\theta^*(L)|_{\zero(\tilde{t})}$ is ample.
		
		\item If $10\leq \dim(X')\leq 19$, the singular locus of $D_Y(t)$ is supported in codimension $5$. Moreover, the preimage of $\Sing(D_Y(t))$ 
		inside $\zero(\tilde{t})$ is a $\PP^2$-bundle over it. This comes from the fact that $\Sing(Y)$ is the space of totally decomposable forms $W=\omega_1 \wedge \omega_2 \wedge \omega_3$; 
		the resolution of $Y$ has fiber over $W$ canonically isomorphic to $\PP(W)$, where $W$ is seen as a vector space of dimension $3$. 
		Therefore, the bundle $\theta^*(L)|_{\zero(\tilde{t})}$ 
		is $2$-ample in this case. 
		
		\item If $\dim(X')\geq 20$, $D_0(t)$ is non-empty in general, and the fiber over it is a $\PP^5$-bundle. In this situation, the bundle $\theta^*(L)|_{\zero(\tilde{t})}$ 
		is $5$-ample.
	\end{itemize}
	As in each case $\dim(\zero(\tilde{s}))-k\geq 2$, by applying Okonek's result we get that the relative homotopy group $\pi_2(\zero(\tilde{t}),\zero(\tilde{s}))$ 
	is trivial. Moreover, $\zero(\tilde{t})$ is an almost Fano variety (\cite{JPR06}), i.e.\ its canonical bundle is big and nef. Almost Fano varieties are simply connected (see \cite{Takayama00}); therefore $\pi_1(\zero(\tilde{t}))$ 
	is trivial. Using the long exact sequence of relative homotopy groups, we deduce that $\pi_1(\zero(\tilde{s}))$ 
	is trivial as well.
\end{proof}

Recall that, for a vector bundle $V$, being $k$-ample means that $\cO_{\PP(V)}(1)$ is $k$-ample. Therefore, the same proof 
remains valid if we replace $L$ by an ample vector bundle $V$, provided that $\dim(\zero(\tilde{s}))\geq 2+k$ whenever $\theta^*(V)|_{\zero(\tilde{t})}$ is $k$-ample.

\subsection{Explicit examples}

\subsubsection{Threefolds}
\label{subsec:threefolds}
We collect here examples of threefolds with trivial canonical bundle which can be constructed as orbital degeneracy loci $D_{Y}(s)$, where $s \in \HHH^0(X, \wedge^3 E)$ is a general section of the globally generated vector bundle $\wedge^3 E$, $E$  a rank 6 vector bundle on a projective variety $X$. 
As in the whole section, $Y$ is the orbit closure of partially decomposable forms in $\wedge^3 
\mathbb{C}^6$.

The relevant varieties $X$ are homogeneous spaces or linear sections of homogeneous spaces. We present them, 
as well as the other varieties that we will meet later on, as zero loci of general sections of a homogeneous vector 
bundle $V$ on a homogeneous variety $X'$. Moreover the bundle $E$ on $X$ will be the restriction of a 
homogeneous bundle $E'$ on $X'$. 

The non-vanishing of the top Chern class of $\wedge^3 {\cal Q}_{\PP(E)}$, which we checked using \Mac\ \cite{M2}, ensures that the constructed degeneracy loci are non-empty.

Using the Koszul complex and the conormal sequence we recover the cohomology groups on $\zero(\tilde{s})$ from those on $\PP(E)$ and, since $\zero(\tilde{s}) \simeq D_Y(s)$, we obtain $\HHH^{p,q}(\zero(\tilde{s})) = \HHH^{p,q}(D_Y(s))$, for $0 \leq p,q \leq 3$. In particular, for all cases we have $\hhh^{1,0}(D_Y(s)) = \hhh^{2,0}(D_Y(s))=0$, hence they are (possibly non-simply connected) Calabi--Yau varieties. In Appendix \ref{appendix} we explain more in detail the method used to compute the Hodge diamonds. 
We list the aforementioned examples of Calabi--Yau threefolds with their Hodge numbers in Table \ref{tab:3folds}.

\newcounter{threefolds}
\newcommand{\three}{\arabic{threefolds}\stepcounter{threefolds}}
\stepcounter{threefolds}
\begin{table}[h!bt]
	\label{tab:cy3folds}
	\begin{center}
		\caption{Some examples of Calabi--Yau $3$-folds}
		\label{tab:3folds}
		\begin{tabular}{cccccc} \toprule
			\phantom{n} & $X'$ & $V$ & $E'$ & $\hhh^{1,1}$ & $\hhh^{2,1}$ 
			\\	\midrule
			$\begin{array}{l}
			\mbox{(t.\three)}  \\
			\mbox{(t.\three)}  
			\end{array}$ & $\Gr(2,7)$ & $2\cO(1)$ & 
			$\begin{array}{c}
			\cU^* \oplus 4\cO \\
			\cQ \oplus \cO
			\end{array}$ &
			$\begin{array}{c}
			2 \\
			3/4?
			\end{array}$ &
			$\begin{array}{c}
			49 \\
			36/37?
			\end{array}$
			\\ \midrule
			$\begin{array}{l}
			\mbox{(t.\three)} 
			\end{array}$ & $\Gr(3,6)$ & $\cO(1)$ & 
			$\begin{array}{c}
			\cU^* \oplus 3\cO \\
			\end{array}$&
			$\begin{array}{c}
			2 
			\end{array}$ &
			$\begin{array}{c}
			38 
			\end{array}$
			\\ \midrule
			$\begin{array}{l}
			\mbox{(t.\three)} \\
			\mbox{(t.\three)} 
			\end{array}$ & $\PP^4 \times \PP^4$ &  & 
			$\begin{array}{c}
			p_1^*\cO(1) \oplus p_2^*\cO(1) \oplus 4\cO \\
			p_1^*\cQ \oplus p_2^*\cO(1) \oplus \cO \\
			\end{array}$&
			$\begin{array}{c}
			3 \\
			4
			\end{array}$ &
			$\begin{array}{c}
			48 \\
			32 
			\end{array}$
			
			\\ \bottomrule
		\end{tabular}
	\end{center}
\end{table}

Notice that, by Proposition \ref{propsimpleconn}, the threefolds (t.1), (t.2), (t.3) are simply connected; for cases (t.4), (t.5) the same proposition cannot be applied.

For the threefold (t.2) the ambiguity in the Hodge numbers cannot be resolved by our method since we could not determine whether one of the coboundary maps of the Koszul complexes has maximal rank. The same happens in example (t.3), where the Picard number can be $1$ or $2$. However, in this case we verify that the two line bundles $L_1=\cO_{\PP(E)}(1)$ and $L_2=\theta^*\cO_X(1)$, where $\theta: \PP(E) \ra X$, are non-trivial and independent by comparing their intersection numbers.

The Hodge numbers in Table \ref{tab:3folds} were previously found e.g.\ in \cite{GHL89}, \cite{KS00}, \cite{KKRS05}, \cite{BK10} as pertaining to complete intersections in toric ambient varieties.
It would be interesting to investigate whether there exists a relation between these examples and ours.

\subsubsection{Fourfolds}
\label{fourfoldssect}

With the same notation as above, we list in Table \ref{tab:4foldsGr} examples of $Y$-degeneracy loci of dimension $4$ with trivial canonical bundle constructed from a pair $(X,E)$, where $X$ is the zero locus of a homogeneous vector bundle $V$ on a classical Grassmannian $X'$. We denoted by $\cT_{+\frac{1}{2}}$, on an orthogonal Grassmannian $\OGr(k,2n)$, one of the two spin bundles of rank $2^{n-k-1}$.

\newcounter{fourfolds}
\newcommand{\four}{\arabic{fourfolds}\stepcounter{fourfolds}}
\stepcounter{fourfolds}
\begin{table}[h!bt]
	\begin{center}
		\caption{Calabi--Yau $4$-folds in classical Grassmannians}
		\label{tab:4foldsGr}
		\begin{tabular}{cccc} \toprule
			\phantom{n} & $X'$ & $V$ & $E'$
			\\ \midrule
			(f.\four) & $\PP^9$ & - & 
			$ 2\cO(1) \oplus 4\cO $
			\\	\midrule
			$\begin{array}{l}
			\mbox{(f.\four)} \\
			\mbox{(f.\four)}  
			\end{array}$ & $\Gr(2,7)$ & $\cO(2)$ & 
			$\begin{array}{c}
			\cU^* \oplus 4\cO \\
			\cQ \oplus \cO
			\end{array}$ 
			\\ \midrule
			$\begin{array}{l}
			\mbox{(f.\four)} \\
			\mbox{(f.\four)}
			\end{array}$ & $\Gr(2,8)$ & 3$\cO(1)$ & 
			$\begin{array}{c}
			\cU^* \oplus 4\cO \\
			\cQ
			\end{array}$ 
			\\ \midrule
			$\begin{array}{l}
			\mbox{(f.\four)} \\
			\mbox{(f.\four)}
			\end{array}$ & $\Gr(2,8)$ & $S^2 \cU^* $ & 
			$\begin{array}{c}
			\cU^* \oplus 4\cO \\
			\cQ
			\end{array}$ 
			\\ \midrule
			$\begin{array}{l}
			\mbox{(f.\four)} \\
			\mbox{(f.\four)}
			\end{array}$ & $\Gr(3,7)$ & $\wedge^2 \cU^* $ & 
			$\begin{array}{c}
			\cU^* \oplus 3\cO \\
			\cQ \oplus 2\cO
			\end{array}$ 
			\\ \midrule
			(f.\four) & $\OGr(2,10)$ & $\cT_{+\frac{1}{2}}(1) $ & 
			$\begin{array}{c}
			\cU^* \oplus 4\cO 
			\end{array}$ 
			\\ \midrule
			(f.\four) & $\OGr(2,12)$ & $\cT_{+\frac{1}{2}}(1) $ & 
			$\begin{array}{c}
			\cU^* \oplus 4\cO 
			\end{array}$ 
			\\ \bottomrule
		\end{tabular}
	\end{center}
\end{table}

The $Y$-degeneracy loci that we obtain from the data $X', V, E'$ of Table \ref{tab:4foldsGr} were 
checked to be non-empty because $\ctop(\wedge^3\cQ_{\PP(E)}) \neq 0$, and Calabi--Yau because the 
Euler characteristic $\chi(D_{Y}(s))=2$. Note that since the dimension is even, this is enough to ensure the 
simple connectedness. 

In cases (f.1--f.9) we  used the package \texttt{Schubert2} implemented in \Mac\ to compute directly the Euler characteristic and the top Chern class of $\wedge^3\cQ_{\PP(E)}$. The same method does not apply for cases (f.10) and (f.11), as orthogonal Grassmannians are not implemented in the package. Instead, we computed directly the dimension of $\HHH^i(\zero(\tilde{s}),\cO_{\zero(\tilde{s})})$ by means of a Koszul complex, as explained in Appendix \ref{appendix}. The same computations show at once the non-emptiness of these loci.

Note that if a triple $(X', V, E')$ satisfies the conditions we require (with the exception of $\PP^9$), then the zero locus of a general
section in $\HHH^0(X',V\oplus 5\cO (1))$ is a fourfold with trivial canonical bundle. Such fourfolds have been
classified in \cite{Benedetti16}, and this classification guarantees that Table \ref{tab:4foldsGr} is complete.

Non-classical generalized Grassmannians may also be considered. For instance, on the Cayley plane $X_{E6.1}$ the zero locus of seven general sections of the positive generator of the Picard group is a Fano variety of dimension nine and index five. Unfortunately, for this case, as well as for the other cases coming from exceptional Lie groups, we could not find any suitable rank six vector bundle $E$. 

\medskip
As discussed in section \ref{grahomvar}, another family of examples is provided by varieties  $X$ defined 
as Grassmann bundles $\Grb(k,F)$, for some vector bundle $F$ on some Fano variety $Z$. In this situation, 
a natural choice for $E$ is $\cU^*_{X/Z}\oplus (6-k)\cO$, where $\cU_{X/Z}$ denotes 
the tautological bundle of $\Grb(k,F)$. As already explained, for every $1\leq k\leq 4$ we know all the possible varieties $Z$ which are suitable to construct degeneracy loci with trivial canonical bundle. The problem is to find suitable bundles on them. Table  \ref{tab:4foldsGrBundle} reports the examples we were able to construct.  
Once again, we did not include in the table the cases in which $E$ decomposes as a line bundle and five copies of the trivial bundle, which happens exactly for $k=1$. 
For Table \ref{tab:4foldsGrBundle} we decided to follow this notation: $Z$ will be the zero locus of a general 
section of a bundle $V$ on a variety $Z'$ (sometimes $Z=Z'$). Moreover $Bl_{pt}\PP^3$ is the blow-up of $\PP^3$ over a point, with exceptional divisor $Exc$.

\begin{table}[h!bt]
	\begin{center}
		\caption{Calabi--Yau $4$-folds in Grassmann bundles}
		\label{tab:4foldsGrBundle}
		\begin{tabular}{ccccc} \toprule
			\phantom{n} & $Z'$ & $V$ & $F$ & $k$ 
			\\ \midrule
			\mbox{(f.\four)} & $\Gr(2,5)$ & $S^2 \cU^*$ & $\cU \oplus 3\cO$ & 2
			\\ \midrule
			$\begin{array}{c}
			\mbox{(f.\four)}\\
			\mbox{(f.\four)}\\
			\mbox{(f.\four)}
			\end{array}$ & $\PP^3$ & - & $\begin{array}{c}
			\wedge^2 \cQ^* \oplus 2\cO\\
			\cO(-1) \oplus \cO(-1)\oplus 4\cO\\
			\cO(-2) \oplus 4\cO
			\end{array}$ & 2
			\\ \midrule
			$\begin{array}{c}
			\mbox{(f.\four)}\\
			\mbox{(f.\four)}
			\end{array}$ & $\PP^4$ & $\cO(3)$ & $\begin{array}{c}
			\cQ^* \oplus \cO\\
			\cO(-1) \oplus 4\cO
			\end{array}$ & 2
			\\ \midrule
			$\begin{array}{c}
			\mbox{(f.\four)}\\
			\mbox{(f.\four)}
			\end{array}$ & $Bl_{pt}\PP^3$ & - & $\begin{array}{c}
			Exc(-1)\oplus \cO(-1)\oplus 3\cO\\
			Exc(-2)\oplus 4\cO
			\end{array}$ & 2
			\\ \midrule
			$\begin{array}{c}
			\mbox{(f.\four)}\\
			\mbox{(f.\four)}\\
			\mbox{(f.\four)}\\
			\mbox{(f.\four)}
			\end{array}$ & $\PP^2\times\PP^2$ & $\cO_1(1)\otimes \cO_2(1)$ & $\begin{array}{c}
			\cQ^*_1\oplus \cQ^*_2\oplus \cO\\
			\cQ^*_1\oplus \cU_2\oplus 2\cO\\
			\cU_1\oplus \cU_2\oplus 3\cO\\
			\cU_1\otimes \cU_2\oplus 4\cO
			\end{array}$ & 2
			\\ \midrule
			$\begin{array}{c}
			\mbox{(f.\four)}\\
			\mbox{(f.\four)}
			\end{array}$ & $\Gr(2,4)$ & $\cO(2)$ & $\begin{array}{c}
			\cU \oplus 3\cO\\
			\cO(-1) \oplus 4\cO
			\end{array}$ & 2
			\\ \midrule
			$\begin{array}{c}
			\mbox{(f.\four)}\\
			\mbox{(f.\four)}\\
			\mbox{(f.\four)}
			\end{array}$ & $\Gr(2,5)$ & $3\cO(1)$ & $\begin{array}{c}
			\cQ^* \oplus 2\cO\\
			\cU \oplus 3\cO\\
			\cO(-1) \oplus 4\cO
			\end{array}$ & 2
			\\ \midrule
			$\begin{array}{c}
			\mbox{(f.\four)}\\
			\mbox{(f.\four)}\\
			\mbox{(f.\four)}
			\end{array}$ & $\PP^1\times \PP^1\times \PP^1$ & - & $\begin{array}{c}
			\cU_1\oplus \cU_2\oplus \cU_3 \oplus 2\cO\\
			(\cU_1\otimes\cU_2)\oplus \cU_3 \oplus 3\cO\\
			(\cU_1\otimes \cU_2\otimes \cU_3) \oplus 4\cO
			\end{array}$ & 2
			\\ \midrule
			$\begin{array}{c}
			\mbox{(f.\four)}\\
			\mbox{(f.\four)}
			\end{array}$ & $\Gr(2,4)$ & $\cO(1)$ & $\begin{array}{c}
			\cU \oplus 3\cO\\
			\cO(-1) \oplus 4\cO
			\end{array}$ & 3
			\\ \midrule
			\mbox{(f.\four)} & $\PP^6$ & $\cO(3)$ & $\cO(-1) \oplus 4\cO$ & 4
			\\ \midrule
			\mbox{(f.\four)} & $\PP^7$ & $2\cO(2)$ & $\cO(-1) \oplus 4\cO$ & 4
			\\ \midrule
			$\begin{array}{c}
			\mbox{(f.\four)}\\
			\mbox{(f.\four)}\\
			\mbox{(f.\four)}
			\end{array}$ & $\Gr(2,5)$ & $\cO(1)$ & $\begin{array}{c}
			\cQ^* \oplus 2\cO\\
			\cU \oplus 3\cO\\
			\cO(-1) \oplus 4\cO
			\end{array}$ & 4
			\\ \bottomrule
		\end{tabular}
	\end{center}
\end{table}

The varieties obtained this way are smooth fourfolds and have trivial canonical bundle. With the package \texttt{Schubert2} we can check that $\wedge^3\cQ_{\PP(E)}$ has non-zero top Chern class and compute the Euler characteristic of the varieties just found: it turns out to be always 2.

\medskip
Besides the examples in Tables \ref{tab:4foldsGr} and \ref{tab:4foldsGrBundle}, many others can be constructed. We might look at different kind of base varieties, or relax some hypotheses we made. Even though a systematic study of these more general cases falls outside the aims of this paper, let us mention here a few sporadic examples.

We can take a more general homogeneous space as $X$, e.g.\ a partial flag variety. Let $X=F(1,5,6)$; we can see it as a codimension 1 complete intersection in $\PP^5 \times \PP^5$ cut out by an equation of bidegree $(1,1)$. Using this description, it is easy to see that for $E$ we can consider the following vector bundles:
\[
\cU_1^* \oplus \cU_2^* \oplus 4 \cO, \qquad \cQ_1 \oplus \cU_2^*.
\]
A computation with \texttt{Schubert2} shows that the corresponding degeneracy loci are non-empty and have characteristic two, hence they are examples of Calabi--Yau fourfolds.

Other fourfolds with trivial canonical bundle can be obtained inside Grassmann bundles over subvarieties of homogeneous varieties, as done above. With the same notation, we can consider a rank 5 vector bundle $F$ over $Z$ such that conditions \eqref{condgrassbundle3} hold. We get the four examples listed in Table \ref{tab:4foldsSporadic}, where $\pi:\Grb(k,F|_{Z}) \rightarrow Z$ is the map associated to the Grassmann bundle.
\begin{table}[h!bt]
	\begin{center}
		\caption{Some other Calabi--Yau $4$-folds in Grassmann bundles}
		\label{tab:4foldsSporadic}
		\begin{tabular}{ccccc} \toprule
			$Z'$ & $G$ & $F$ & $k$ & $E$
			\\ \midrule
			$\PP^5$ & - & $\begin{array}{c}
			\cO(-1) \oplus 4\cO\\
			\cQ^*
			\end{array}$ & 1 & $\cU_{X/Z}^* \oplus \pi^*L^* \oplus 4\cO$
			\\ \midrule
			$\PP^6$ & $\cO(2)$ & $5\cO$ & 1 & $\begin{array}{c}
			\cU_{X/Z}^* \oplus \pi^*L^* \oplus 4\cO
			\\
			\cQ_{X/Z} \oplus \pi^*L^* \oplus \cO
			\end{array}$
			\\ \bottomrule
		\end{tabular}
	\end{center}
\end{table}

The last two examples correspond to degeneracy loci inside $\QQ^5\times \PP^4$, where $\QQ^5$ denotes the five-dimensional quadric. For all four examples, a computation with \texttt{Schubert2} shows that the corresponding degeneracy loci are non-empty Calabi--Yau fourfolds.

A last example which is worth recalling here is the twisted degeneracy locus constructed from $\PP^9$ (and mentioned in Section \ref{twisteddeg}). Again, the Euler characteristic in this case is equal to two.

\section{Nilpotent orbits} 
\label{nilpotOrbits}

In this section we study degeneracy loci associated to Richardson nilpotent orbits. We give a list of orbits which can be used to construct low-dimensional degeneracy loci. These loci will often have singularities in low codimension. Nonetheless, their resolutions of singularities give rise to many examples of threefolds and fourfolds with trivial canonical bundle.

\subsection{A reminder about nilpotent orbits}
Consider any projective homogeneous variety $G/P$ and take
as $\cW$ the cotangent bundle $\Omega^1_{G/P}$. Then condition \eqref{condCrepancy} is obviously verified. 
Note that $\Omega^1_{G/P}$ is the homogeneous vector bundle defined by the $P$-module $(\fg / \fp)^*
\simeq \fp^{\perp}\subset \fg^*$. If we identify $\fg^*$ with $\fg$ using the Killing form, then 
$\fp^{\perp}=\rad(\fp)$, the nilpotent radical of the Lie algebra $\fp$. The image of the map 
$$p_W : \cW\longrightarrow Y\subset \fg$$
is therefore contained in the nilpotent cone, so that $Y=\overline\cO$ is the closure of some nilpotent 
orbit $\cO$. Such orbits are called {\it Richardson orbits}. The main example is of course the maximal 
nilpotent orbit; in this case $P=B$ is a Borel subgroup, $Y$ is the nilpotent cone, and $p_W$ is the 
famous Springer resolution.  Nevertheless, $p_W$ is not necessarily birational in general; 
Fu proved in \cite{Fu03} that this is the case exactly when  $\overline\cO$ admits a symplectic resolution (moreover
this resolution must be some $p_W$). 

Finally, a general useful fact about nilpotent orbit closures is that the 
singular locus of   $Y=\overline\cO$ always coincides with its boundary $\overline\cO - \cO$.

\subsection{Associated degeneracy loci}
\label{assdegloc}
In the relative setting, we start from a $G$-principal bundle $\cE$ over some variety $X$, and we denote by $\ad_\cE$ the vector 
bundle $\cE_{\fg}$ on $X$ associated to the adjoint representation of $G$. Let $\cO$ be a Richardson nilpotent orbit in $\fg$, corresponding to a parabolic subgroup $P$ of $G$. As done in Section \ref{twisteddeg}, we can consider twisted degeneracy loci: let $L$ be a line bundle on $X$ such that  $\ad_\cE\otimes L$ is generated by global sections. For $s$ such a global section, 
the $\overline\cO$-degeneracy locus is 
$$D_{\overline\cO}(s)=\{x\in X, \; s(x)\in \cE_{\overline\cO}\otimes L\}.$$
If the collapsing of $\Omega^1_{G/P}$ is birational and $s$ is general, this locus will be desingularized by $\zero(\tilde{s})$, $\tilde{s}$ being the section of the vector bundle 
$$Q=\theta^*(\ad_\cE)/\Omega^1_{\cE_{G/P}/X}\otimes \theta^*L$$
induced by $s$. Since $\ad(\cE)$ is self-dual, its determinant is trivial (at least up to $2$-torsion, something we will ignore in the sequel since we will always work with varieties whose Picard group has no torsion). On the one hand,  we get the simple formula
$$K_{\zero(\tilde{s})}=\theta^*(K_X\otimes L^{\dim P})|_{\zero(\tilde{s})}.$$
On the other hand, the dimension of $\zero(\tilde{s})$ is
$$\dim \zero(\tilde{s})=\dim X-\ell_P,$$
where $\ell_P$ denotes the dimension of the Levi part of $P$, which can be computed as $\ell_P=2\dim P-\dim G$. 

If we require, for example, $\zero(\tilde{s})$ to be of dimension $d$
with trivial canonical bundle, we need the index of $X$ to be 
equal to $\dim P$ (or a multiple, if $L$ is divisible), while its
dimension must be $\dim X=d+\ell_P$. This yields the relation $$\ind(X)=\dim X-d+\dim G/P.$$
Because of \eqref{KOin}, this implies that
\begin{equation}
	\label{restrictionOrbit}
	\dim G/P\le d+1.
\end{equation}
Moreover, in case of equality $X$ must be a projective space, while if $\dim G/P= d$, then $X$ must be a quadric.

\subsection{Nilpotent orbits in type A}
If $\fg=\fsl_e$, every nilpotent orbit is a Richardson orbit, and admits a symplectic resolution. 
Nilpotent orbits are in bijective correspondence with partitions of $e$, the parts of the partition 
being the sizes of the Jordan blocks. Let us denote 
by $\cO_{\lambda}$ the nilpotent orbit associated to the partition $\lambda$ of $e$.
Symplectic resolutions of $\overline{\cO_{\lambda}}$ are given by the cotangent bundles of the flag 
varieties $F_{\underline{d}}$, where the sequence $\underline{d}=(\lambda^*_{\sigma(1)}, 
\lambda^*_{\sigma(1)}+\lambda^*_{\sigma(2)}, \ldots )$ for some permutation $\sigma$, and $\lambda^*$
is the partition dual to $\lambda$, i.e.\ $\lambda^*_i$ is the number of parts of $\lambda$ which are greater than or equal to $i$. Hence, a given orbit closure has in general several non-equivalent 
symplectic resolutions, being Richardson with respect to different types of parabolic subgroups.

Inside $\fsl_e$, an orbit $\cO_{\mu}$
is contained in the closure of $\cO_{\lambda}$ if and only if $\mu\le\lambda$ with respect to the 
dominance order, which means that $\mu_1+\cdots +\mu_i\le \lambda_1+\cdots +\lambda_i$ for all $i$. 
So the irreducible components of the singular locus of  $\overline\cO_{\lambda}$ are the orbit closures 
$\overline\cO_{\mu}$, where $\mu$ is obtained from $\lambda$ by moving a corner of the diagram of $\lambda$ down to the first possible lower row; the codimension is 
then twice the difference of rows between the initial and final 
positions of the corner that has been moved. An easy consequence is that the codimension of the 
singular locus is at least four exactly when $\lambda_i-\lambda_{i+1}\in \{0,1\}$ for all $i$.

\medskip
In the relative setting, we consider a vector bundle $E$ of rank $e$ on $X$, and a line bundle $L$. 
For a morphism $\varphi : E\lra E\otimes L$, and a partition $\lambda$ of $e$, we consider the locus 
$D_{\lambda}(\varphi)$ of points $x\in X$ where the traceless part of $\varphi_x$ is nilpotent 
of Jordan type $\lambda$, or more degenerate. When $\End(E)\otimes L$ is globally generated, and 
$\varphi$ is general, a birational model of $D_{\lambda}(\varphi)$ is the zero-locus $\zero(\tilde\varphi)$
of the corresponding section of  $\tilde\varphi$ of $\theta^*(\End(E)\otimes L)/\cW_E$ on the relative
flag variety $F_{\underline{d}}(E)$, where  $\cW_E$ is the 
relative cotangent bundle, twisted by $L$. If we denote by $d(\lambda)$ the relative dimension of 
$F_{\underline{d}}(E)$ (which depends only of $\lambda$), we deduce that 
$$K_{\zero(\tilde\varphi)}=\theta^*(K_X\otimes L^{e^2-1-d(\lambda)})|_{\zero(\tilde\varphi)}.$$
Moreover the codimension of $D_{\lambda}(\varphi)$ in $X$ is equal to the dimension of the Levi part of 
the parabolic, that is, 
$$\dim D_{\lambda}(\varphi)=\dim X-\sum_i(\lambda^*_i)^2+1.$$

Consider for example the minimal orbit closure $Y_{\min}$ in $\fsl_e$. This is the closure of the orbit of nilpotent 
endomorphisms of rank one, whose projectivization is the flag variety $F_{1,e-1}(\CC^e)$. This orbit 
closure $Y_{\min}$ has two symplectic resolutions, by the cotangent bundles of $\PP(\CC^e)$ and its dual. 
In the relative setting we get the formulas
$$K_{\zero(\tilde\varphi)}=\theta^*(K_X\otimes L^{e(e-1)})|_{\zero(\tilde\varphi)}, \qquad 
\dim D_{\lambda}(\varphi)=\dim X-(e-1)^2.$$

Consider finally the maximal orbit closure $Y_{\max}$ in $\fsl_e$. This is the full nilpotent cone, 
and its unique symplectic resolution is the Springer resolution by the cotangent bundle of the 
full flag variety.  
In the relative setting we get the formulas
$$K_{\zero(\tilde\varphi)}=\theta^*(K_X\otimes L^{e(e+1)/2})|_{\zero(\tilde\varphi)}, \qquad 
\dim D_{\lambda}(\varphi)=\dim X-(e-1).$$

\subsection{$G_2$-structures}
Recall that $G_2$ can be defined as the stabilizer of a generic skew-symmetric three-form in seven variables. 
More precisely, there is a degree seven $SL_7$-invariant polynomial $P$ on $\wedge^3(\CC^7)^*$ such that a three-form
on which $P$ does not vanish has a stabilizer isomorphic to $G_2$. This implies that a $G_2$-principal bundle on $X$ can be defined from a rank seven bundle $E$ on $X$, with a global three-form $\omega : \wedge^3E\ra L$, for some line 
bundle $L$, such that the induced map $P(\omega) : (\det E)^3\ra L^7$ is an isomorphism. 

By reduction to $\fsl_3$, one way to do that would be to start with a rank three vector bundle $F$ with 
trivial determinant. Let $\alpha : \wedge^3F\ra \cO_X$ and $\alpha^* : \wedge^3F^*\ra \cO_X$ be some trivializations. 
Then the rank seven vector bundle $E=F\oplus \cO_X\oplus F^*$ defines a $G_2$-structure on $X$: indeed, there is a natural three-form $\omega$ on $E$ defined by the composition 
\[\xymatrix{
	\wedge^3E \ar[r] & \wedge^3F\oplus (F\otimes\cO_X\otimes F^*)\oplus \wedge^3F^* \ar[r]^-\mu & \cO_X,
}
\]
where $\mu =\alpha\oplus id_F\oplus \alpha^*$. 
This three-form is everywhere non-degenerate, i.e.\ $P(\omega)$ does not vanish. 
In this setting the adjoint bundle is 
$$\ad_{\fg_2}(E)=\ad(F)\oplus F^*\oplus F.$$

\subsection{Examples of small dimension}
For the construction of varieties with trivial canonical bundle up to dimension $d=4$, condition \eqref{restrictionOrbit} leaves only few possibilities, which we compile in Table \ref{table:richardson}. In such table $\QQ^n$ denotes the $n$-dimensional quadric, while $F_{\underline{d}}$ (resp.\ $OF_{\underline{d}}$) denotes (partial) flag varieties (resp.\ of isotropic subspaces with respect to a non-degenerate symmetric form). The integer $\delta$ in the last column is the degree of the map $\Omega^1_{G/P}\ra\overline\cO$; it is always equal to 
one in type A or for the Springer resolutions (cases (6) and (9)). It is easy to check that its value is two for 
odd dimensional projective spaces, considered as homogeneous varieties for symplectic groups 
(cases (4) and (11)). For cases (5) and (12) see \cite[Proposition 3.21]{Fu03}; cases (15) and (16) for $G_2$ are discussed in \cite[Lemma 5.4 and Appendix]{Fu07}. The closure of each Richardson orbit listed in Table \ref{table:richardson} is normal and has rational singularities, see \cite{KraftProcesi82,Kraft89}.

\begin{table}[h!bt]
	\begin{center}
		\caption{Some Richardson orbits}
		\label{table:richardson}
		\begin{tabular}{ccccccc} \toprule
			& $\dim G/P$ & $G/P$ & $G$ & $\ell_P$ & $\codim 
			(\Sing \overline\cO)$ & $\delta$
			\\ \midrule
			(1) & 1 & $\PP^1$ & $SL_2$ &1& 2&1
			\\ 
			(2) & 2 & $\PP^2$ & $SL_3$ &4& 4&1
			\\ 
			(3) & 3 & $\PP^3$ & $SL_4$ &9& 6&1
			\\ 
			(4) &   & $\PP^3$ & $Sp_4$ &4& 2&2
			\\ 
			(5) &   & $\QQ^3$ & $SO_5$ &4& 2&1
			\\ 
			(6) &   & $F_{1,2}$ & $SL_3$ &2& 2&1
			\\ 
			(7) & 4 & $\PP^4$ & $SL_5$ &16& 8&1
			\\ 
			(8) &   & $\QQ^4$ & $SL_4$ &7& 2&1
			\\ 
			(9) &   & $OF_{1,2}$ & $SO_5$ & 2& 2&1
			\\ 
			(10) & 5 & $\PP^5$ & $SL_6$ &25& 10&1
			\\ 
			(11) & & $\PP^5$ & $Sp_6$ &11& 2&2
			\\ 
			(12) & & $\QQ^5$ & $SO_7$ &11& 2&1
			\\ 
			(13) & & $F_{1,2}$ & $SL_4$ &5& 2&1
			\\ 
			(14) & & $F_{1,3}$ & $SL_4$ &5& 2&1
			\\ 
			(15) & & $\QQ^5$ & $G_2$ &4& 2&2 
			\\ 
			(16) & & $G_2^{\ad}$ & $G_2$ &4& 2&1
			\\ \bottomrule
		\end{tabular}
	\end{center}
\end{table}

\subsubsection{Threefolds}
If we want to construct threefolds with trivial canonical bundle, then we can use cases (1) to (9). In cases (3)-(6), the base variety $X$ must be a quadric of dimension $\ell_P+3$, and in cases (7)-(9), a projective space of this dimension. The line bundle $L$ must be the generator of the Picard group and the principal bundle can always be chosen to be the trivial one. But other choices are 
possible; if the structure group is $G=SL_e$ we need a rank $e$ vector bundle $E$ such that $\ad(E)\otimes L$ is generated by global sections, and $E=k\cO\oplus (e-k)\cO(1)$ is always a solution (by symmetry we may suppose that $2k\le e$).
If the structure group is $G=Sp_4$ or $G=SO_5$ (recall the exceptional isomorphism  $Sp_4\simeq Spin_5$, by which $\QQ^3 \cong \IGr(2,4)$, the symplectic Grassmannian of isotropic $2$-planes with respect to a non-degenerate skew-symmetric form), we need a vector bundle of rank $4$ or $5$ with an everywhere non-degenerate 
bilinear form, possibly with values in a line bundle. For 
$G=Sp_4$ we can choose $E=2\cO\oplus 2\cO(1)$, but 
we found no non-trivial solution for $SO_5$. In 
case (2), the base variety $X$ must be of dimension $7$ and index $6$, hence a del Pezzo manifold. 
For example it could be a cubic hypersurface in $\PP^8$ or the intersection of two quadrics in $\PP^9$. 
In case (1), $X$ must be of dimension $4$ and index divisible by $2$, so essentially a Mukai variety.

\subsubsection{Fourfolds}
If we want to construct fourfolds with trivial canonical bundle, we can also use cases (10)-(16), for which the base variety $X$ must be a projective space of dimension $\ell_P+4$, and cases (7)-(9), with a quadric of this dimension. 
Note that cases (13) and (14) correspond to two different desingularizations of the same nilpotent orbit. 
For cases (3)-(6), we need a base variety $X$ of coindex two.

Apart from complete intersections, for case (6) we can use
$X=\Gr(2,5)$. We have then several additional choices for our bundle $E$, which can be
$$E=3\cO, \quad 2\cO\oplus\cO(-1), \quad  
\cU\oplus\cO, \quad \cU\oplus\cO(-1), \quad \cQ.$$ 
All of these fourfolds turn out to have Euler characteristic $\chi(\cO_{D_Y(s)})=2$, as a direct computation in \Mac\ shows, hence are Calabi--Yau varieties.

For case (2) we need a variety $X$ of dimension $8$ and index $6$, and apart from complete intersections we can
choose $X=\Gr(2,6)$ and $E$ one of the bundles
$$E=3\cO, \quad 2\cO\oplus\cO(-1), \quad  
\cU\oplus\cO, \quad \cU\oplus\cO(-1).$$ 
In this case the orbital degeneracy locus has only isolated singularities; their resolutions have characteristic two as well.

\section{Fano degeneracy loci}
\label{fanoDegLoci}

In this section we exhibit some Fano and almost Fano varieties obtained as orbital degeneracy loci, or resolutions thereof. The case of threefolds is pretty interesting: by computing their invariants (for instance, by means of \Mac), the existing complete classifications (see \cite{IP99}) will allow us to identify them
explicitely.  

In the case of the subvariety $\overline{Y_2} \subset \wedge^3\CC^6$ of partially decomposable forms, studied in Section \ref{partDecForms}, the equation to be satisfied in order to construct a Fano variety is 
\[
K_X=(\det E)^{-5}\otimes L
\]
where $L$ is a line bundle whose dual is ample. In this way,
$K_{D_Y(s)}=L|_{D_Y(s)}$. If we try to find threefolds (resp.\ fourfolds), one possibility is to require the variety $X$ to be of index $6$ and dimension $8$ (resp.\ $9$). As in the Calabi--Yau case, we can look for such $X$ among subvarieties of homogeneous spaces. 

Similarly, for nilpotents orbits the restriction  \eqref{restrictionOrbit} on the dimension $d$ of the degeneracy locus given by 
the Kobayashi--Ochiai inequality becomes
\[
\dim G/P \leq d.
\]
In all cases, the line bundle $L$ we use to twist our nilpotent degeneracy loci will necessarily be $\cO_X(1)$. Notice that for Fano varieties one more issue arises if the degeneracy locus is singular, more precisely if the codimension of the singularities of the corresponding orbit closure is smaller than or equal to $d$. Then its resolution will 
not be Fano, but only almost Fano \cite{JPR06}, in the sense that the anticanonical bundle is nef and big.

\begin{remark}
	Suppose that $Y$ has rational singularities and that $D_Y(s)$ is a Fano degeneracy locus of dimension three. Recall that by Proposition \ref{Gorcansing} $D_Y(s)$ is Gorenstein and has canonical singularities: we are therefore in the hypotheses of \cite[Theorem 8.3]{JPR06}. The crepant resolution $\theta':\zero(\tilde{s})\to D_Y(s)$ is in fact the morphism from $\zero(\tilde{s})$ to its anticanonical model. In addition to that, in all the cases we consider, our orbital degeneracy loci $D_Y(s)$ will be anticanonically embedded.
\end{remark}

\subsection{Fano threefolds}
If we want to construct smooth threefolds, only cases (2) and (3) remain, and the only possibilities for $X$ are respectively the $7$-dimensional quartic $\QQ^7$ and $\PP^{12}$.

In Table \ref{tab:3foldsFano} we collect the examples of Fano threefolds $F$ that we constructed as 
orbital degeneracy loci, and the model they correspond to, found using the existing classifications. For each case it is sufficient to compute $(-K_F)^3$ and $\chi(\Omega^1_F)$ to identify the variety.

\begin{table}[h!bt]
	\begin{center}
		\begin{footnotesize}
			\caption{Some Fano degeneracy loci $F$ of dimension $3$}
			\label{tab:3foldsFano}
			\begin{tabular}{cccc} \toprule
				& X & E & Model
				\\ \midrule
				$\overline{Y_2}\subset \wedge^3\mathbb{C}^6$ & $\Gr(2,6)$ & ${\cal U}_X^* \oplus 4{\cal O}_X$ & $\begin{array}{c}
				\mbox{Blow-up of $\PP^3$ along a curve} \\
				\mbox{of degree $7$ and genus $5$}
				\end{array}$
				\\	\midrule
				$\overline{Y_2}\subset \wedge^3\mathbb{C}^6$ & $\Gr(2,6)$ & ${\cal Q}_X\oplus 2{\cal O}_X$ & $\begin{array}{c} 
				\mbox{Blow-up of $F(1,2,3)$ along an} \\
				\mbox{elliptic curve which is an intersection} \\
				\mbox{of two divisors from $|-\frac{1}{2}K_{F(1,2,3)}|$ 
				}  \end{array}$
				\\	\midrule
				(2) $\PP^2, SL_3$ & $\QQ^7$ & $3{\cal O}_X$ & Divisor  of bidegree $(2,2)$ in $\PP^2\times \PP^2$
				\\	\midrule
				(2) $\PP^2, SL_3$ & $\QQ^7$ & ${\cal O}_X(-1)\oplus 2{\cal O}_X$ & Intersection of three quadrics in $\PP^6$
				\\	\midrule
				(3) $\PP^3, SL_4$ & $\PP^{12}$ & $4{\cal O}_X$ & $\begin{array}{c}
				\mbox{Blow-up of $\PP^3$ along a curve} \\
				\mbox{of degree $6$ and genus $3$}
				\end{array}$
				\\	\midrule
				(3) $\PP^3, SL_4$ & $\PP^{12}$ & ${\cal O}_X(-1)\oplus 3{\cal O}_X$ & Intersection of three quadrics in $\PP^6$
				\\ \bottomrule
			\end{tabular}
		\end{footnotesize}
	\end{center}
\end{table}

\begin{remark}
	As in the case of partially decomposable forms, for nilpotent orbits some choices for $E$ give rise to empty loci or complete intersections. A case by case study falls outside the aims of the paper, but as an example we give the following, arising when $Y$ is the orbit of nilpotent matrices of rank $1$ under the action of $SL_n$.
	
	Let us suppose that 
	$E=(n-j)\cO_X\oplus j\cO_X(-1)$, with $j\geq 2$, $n-j\geq j$, and 
	$L=\cO_X(1)$. Then a $j\times (n-j)$ block of the matrix representing the section $s$ is constant on the variety $X$. As $s$ is general, the matrix has at least rank $j$, and if $j\ge 2$ this implies that $D_Y(s)$ is empty.
	Similarly if $E=(n-1)\cO_X\oplus \cO_X(-1)$, then it can be seen that $D_Y(s)$ is just the zero locus of $(n-2)(n-1)$ sections of $\cO_X(1)$ and $(n-1)$ sections of $\cO_X(2)$. This is coherent with what we have obtained in Table \ref{tab:3foldsFano}.
\end{remark}

\subsection{Almost Fano threefolds}
Let us now consider the case of almost Fano threefolds, which will be constructed from nilpotent orbit closures that are 
singular in codimension two. They are  listed in Table \ref{3foldsalmostFano}; the subscripts denote the degree of the complete intersection in the ambient space. The relevant orbits are those  labeled (1), (4), (5), (6) in Table \ref{table:richardson}. The case (4) is particular, as it is the only one for which $\theta':\zero(\tilde{s}) \rightarrow D_Y(s)$ is finite but not birational. In case (1) the variety $X$ has to be a del Pezzo fourfold, 
which means that the index is equal to three, and a complete classification is available (see for instance \cite[Theorem 3.3.1]{IP99}). In case (4) and (5) the variety $X$ is $\PP^7$ and in case (6) it is $\PP^5$.

Notice that the orbit closures (1) and (6) are the full nilpotent cones in the respective Lie algebras. For them the degeneracy locus is well understood, as explained in the following remark.
\begin{remark}
	\label{nilpconefano}
	Let ${\cal N}$ be the nilpotent cone in the simple Lie algebra $\fg$. Since ${\cal N}$ is a complete intersection in $\fg$ (see \cite{ko63}), the degeneracy locus $D_{\cal N}(s)$ is also a (possibly singular) complete intersection of hypersurfaces defined by (non-generic) sections of $L^d$, where $d$ belongs to the set of fundamental exponents of $\fg$. In particular for the group $SL_n$, $D_{\cal N}(s)$ is defined by the vanishing of the coefficients of the characteristic polynomial of the matrix describing $\cE_{\mathfrak{sl}_n}\otimes L$.
\end{remark}
For the nilpotent cone in $SL_2$, $D_{\cal N}(s)$ is the zero locus of $\det(s)\in \HHH^0(X,L^2)$. Similarly, for $SL_3$, $D_{\cal N}(s)$ is the intersection of the zero locus of a section of $L^2$ and a section of $L^3$ (again $\det(s)$). Therefore, in both these cases, the almost Fano threefold $D_Y(s)$ is a degeneration of a smooth Fano threefold which is a complete intersection. These varieties have already been studied, for example see \cite{JPR06}.
The only ambiguity among these cases is the model of the one that is constructed inside $X=\PP(2,1,1,1,1,1)_4$, a quartic hypersurface in the weighted projective space $\PP(2,1,1,1,1,1)$.

\begin{proposition}
	Let $X=\PP(2,1,1,1,1,1)_4$. Denote by $D_Y(s)$ the almost Fano threefold constructed from a bundle $E$ of rank two over $X$ using the orbit closure $Y$ of nilpotent matrices in $\mathfrak{sl}_2$ (orbit (1) in Table \ref{table:richardson}). Then:
	\begin{itemize}
		\item if $E=2\cO_X$, $D_Y(s)$ is a double cover of a quadric $W$ in $\PP^4$ ramified along a its intersection with a quartic;
		\item if $E=\cO_X\oplus \cO_X(1)$, $D_Y(s)$ is a quartic in $\PP^4$.
	\end{itemize}
\end{proposition}

\begin{proof}
	We can suppose that $X\subset \PP(2,1,1,1,1,1)$ is defined by the quartic $P=x_0^2+P_4(x_1,...,x_5)$, where $P_4$ is a polynomial of degree $4$. By projecting on the last five coordinates, $X$ is realized as a double cover of $\PP^4$ ramified along the quartic $\{P_4=0\}$. Moreover, by Remark \ref{nilpconefano} and what follows, $D_Y(s)$ is the zero locus of $\det(s)\in \HHH^0(X,\cO_X(2))$. 
	
	If $E=2\cO_X$, the entries of the matrix representing $s$ are sections of $\cO_X(1)$, i.e.\ polynomials in the variables $x_1,\dotsc,x_5$. Therefore, $\det(s)$ has the form $Q=Q_2(x_1,...,x_5)$ for $Q_2$ a polynomial of degree $2$; as a consequence, $D_Y(s)$ is the double cover of $W=\{Q_2=0\}\subset \PP^4$ ramified along $\{P_4=0\}$.
	
	If $E=\cO_X\oplus \cO_X(1)$, one entry of the matrix representing $s$ is a section of $\cO_X(2)$. Therefore, $\det(s)$ has the form $P'=x_0+P_2(x_1,...,x_5)$, where $P_2$ has degree $2$. This implies that $D_Y(s)$ is actually the quartic in $\PP^4$ defined by the equation $P_4=P_2^2$.
\end{proof}

A little bit more involved is the case of the orbit (5). As already mentioned, $X=\PP^7$, and we have (at least) two choices for $E$ of rank $4$ (we use the isomorphism $Sp_4 \simeq Spin_5$), i.e.\ $E=4\cO_X$ or $E=2\cO_X\oplus 2\cO_X(1)$. In both cases we could compute the degree of $D_Y(s)$ with respect to the anticanonical bundle using \Mac: it is equal to $10$ and $8$, respectively. We guess that $D_Y(s)$ should have an interpretation similar to the one for the other almost Fano degeneracy loci of the same degrees that appear in Table \ref{3foldsalmostFano}.

The degeneracy loci $D_Y(s)$ constructed from the orbit (4) are exactly the same as those constructed from the orbit (5), as for both of them $Y$ is the closure of the subregular nilpotent orbit in $\mathfrak{sp}_4$ (see e.g.\ \cite{co93}); in this case however the morphism $\theta':\zero(\tilde{s})\to D_Y(s)$ is of degree $2$ rather than birational. When $E=4\cO_X$, $\zero(\tilde{s})$ is of degree $20=2\cdot 10$, and when $E=2(\cO_X\oplus \cO_X(1))$ it is of degree $16=2 \cdot 8$, as one would expect. However, we computed $\chi(\cO_{\zero(\tilde{s})})=2$ in the latter case, which seems to indicate that $\zero(\tilde{s})$ splits into two connected components, each isomorphic to the desingularization of $D_Y(s)$ given by case (5).

\begin{table}[h!bt]
	\begin{center}
		\setlength{\tabcolsep}{2.5pt}
		\begin{footnotesize}
			\caption{Some almost Fano degeneracy loci $F$ of dimension $3$}
			\label{3foldsalmostFano}
			\begin{tabular}{ccccc} \toprule
				& X & E & $(-K_F)^3$ & Model
				\\ \midrule
				(1) $\PP^1, SL_2$ & $\PP^2\times \PP^2$ & $\begin{array}{c}
				2\cO_X \\
				\cO_X\oplus \cO(1,1) \\
				\cO(1,0)\oplus \cO(0,1)
				\end{array}$ & $12$ & $(\PP^2\times \PP^2)_{2}$
				\\ \midrule
				(1) $\PP^1, SL_2$ & $\Gr(2,5)_{1^2}$ & $\begin{array}{c}
				2\cO_X \\
				\cO_X\oplus \cO(1) \\
				{\cal U}_{X}^* 
				\end{array}$ & $10$ & $\Gr(2,5)_{1^2,2}$
				\\ \midrule
				(1) $\PP^1, SL_2$ & $\PP^6_{2^2}$ & $\begin{array}{c}
				2\cO_X \\
				\cO_X\oplus \cO(1)
				\end{array}$ & $8$ & $\begin{array}{c}
				\PP^6_{2^3} \\
				\mbox{\cite[Prop.\ 8.10]{JPR06}}
				\end{array}$
				\\	\midrule
				(1) $\PP^1, SL_2$ & $\PP^5_3$ & $\begin{array}{c}
				2\cO_X \\
				\cO_X\oplus \cO(1)
				\end{array}$ & $6$ & $\begin{array}{c}
				\PP^5_{3,2} \\
				\mbox{\cite[Prop.\ 8.10]{JPR06}}
				\end{array}$
				\\	\midrule
				(1) $\PP^1, SL_2$ & $\PP(2,1,1,1,1,1)_4$ & $\begin{array}{c}
				2\cO_X \\
				\cO_X\oplus \cO(1)
				\end{array}$ & $4$ & $\begin{array}{c}
				\mbox{\cite[Prop.\ 8.9]{JPR06}} \\
				\mbox{\cite[Prop.\ 8.10]{JPR06}}
				\end{array}$ 
				\\	\midrule
				(1) $\PP^1, SL_2$ & $\PP(3,2,1,1,1,1)_6$ & $\begin{array}{c}
				2\cO_X \\
				\cO_X\oplus \cO(1)
				\end{array}$ & $2$ & \mbox{\cite[Prop.\ 8.9]{JPR06}}
				\\	\midrule
				(5) $\IGr(2,4), Sp_4$ & $\PP^{7}$ & $\begin{array}{c}
				4\cO_X \\
				2\cO_X\oplus 2\cO(1)
				\end{array}$ & $\begin{array}{c}
				10 \\
				8
				\end{array}$ & ?
				\\	\midrule
				(6) $F_{1,2}, SL_3$ & $\PP^{5}$ & $\begin{array}{c}
				3\cO_X \\
				2\cO_X\oplus \cO(1)
				\end{array}$ & $6$ & $\begin{array}{c}
				\PP^5_{3,2} \\
				\mbox{\cite[Prop.\ 8.10]{JPR06}}
				\end{array}$
				\\ \bottomrule
			\end{tabular}
		\end{footnotesize}
	\end{center}
\end{table}

Finally, we describe the morphism $\theta': \zero(\tilde{s}) \to D_Y(s)$. 

\begin{proposition}
	\label{divcontract}
	For all the cases considered in Table \ref{3foldsalmostFano}, the desingularization $\theta': \zero(\tilde{s}) \to D_Y(s)$ is a divisorial contraction. 
\end{proposition}

\begin{proof}
	Let us study $\theta'^{-1}(C)$, where $C:=\Sing(D_Y(s))=D_{\Sing (Y)}(s)$ (see Proposition \ref{propressing}). Let $Y':=\Sing Y$. We analyze the situation case by case.
	
	Orbit (1). $Y'$ is the $0$-orbit, and $C=\zero(s)$. If $x\in C$,  the whole fiber over $x$ of the morphism $\theta: \PP(E) \to X$ is contained in $\zero(\tilde{s})$. Therefore $\theta'^{-1}(C)$ is a $\PP^1$-bundle over $C$, and $\theta'$ is divisorial.
	
	Orbit (5). $Y$ is the closure of the orbit of nilpotent matrices in $\mathfrak{so}_5$ 
	of rank $2$, while $Y'$ is the closure of the orbit of matrices of rank $2$ whose image $P$ is isotropic. Consider the resolution $p_W:\Omega^1_{\QQ^3}\to Y$. 
	Over $Y \setminus Y'$ it is an isomorphism whose inverse is given by
	\[
	Y\setminus Y' \to \Omega^1_{\QQ^3}\mbox{ , }y\mapsto (l, \phi)
	\]
	where $l\in \im(y)$ is isotropic, and $\phi\in \Hom(l^{\perp}/l, l)$.
	Moreover, $p_W^{-1}(Y')$ 
	is a $\PP^1$-bundle over $Y'$: indeed, the fiber over a point $y\in Y'$ is isomorphic to the locus of isotropic lines in $\im(y)$, which is $\PP(\im(y))\cong \PP^1$  since $\im(y)$ is isotropic. Therefore, in the relative case one gets that $\theta'^{-1}(C)$ is a $\PP^1$-bundle over $C$, and again $\theta'$ is divisorial.
	
	Orbit (6). $Y'$ is the closure of the orbit of nilpotent matrices of rank $1$, whose desingularization is given by the total space of the cotangent bundle of $\PP^2$ and induces a desingularization $\zero_1(\tilde{s}) \rightarrow C$. But since $C$ is one-dimensional, it is smooth and $C \cong \zero_1(\tilde{s}) \subset \PP(E)$. 
	The morphism $\theta: F_{1,2}(E) \to X$ 
	factors through  $\theta_1: \PP(E) \to X$, 
	i.e. $\theta=\theta_1\circ p$, 
	where $p: F_{1,2}(E)\to \PP(E)$ is the natural projection. 
	With this notation, $\theta'^{-1}(C)=p^{-1}(\zero_1(\tilde{s}))$. 
	If $(x,l)\in \zero_1(\tilde{s})$, its preimage under $p$ is given by $\{(x,l,P)\in F_{1,2}(E), \; l\subset P \}$. 
	This implies again that $\theta'^{-1}(C)$ is a $\PP^1$-bundle over $C$, and $\theta'$ is a divisorial contraction.
\end{proof}

Since for the orbits (1) and (6) $D_Y(s)$ is a (singular) complete intersection, its Picard number is the same as the ambient space. When it is equal to $1$ (in all cases except for $X=\PP^2\times \PP^2$), $\zero(\tilde{s})$ is the blow-up of $D_Y(s)$ along the curve $C$ (see for example \cite[Proposition 8.11]{JPR06}).

\subsection{Fano fourfolds}
Finally, in Table \ref{tab:4foldsFano}, we collect a few examples of Fano fourfolds $F$ that can be constructed as orbital degeneracy loci. It is interesting to notice that their invariants do not appear in the classification given in \cite{Kuechle95} for zero loci of sections of homogeneous vector bundles, meaning that the varieties we found are not included in that list. As before, we restricted ourselves to the smooth case. In the case of nilpotent orbits, i.e.\ cases (3) and (7) of Table \ref{table:richardson}, the variety $X$ is forced to be $\QQ^{13}$ and $\PP^{20}$ respectively.

\begin{table}[h!bt]
	\begin{center}
		\begin{small}
			\caption{Some Fano degeneracy loci $F$ of dimension $4$}
			\label{tab:4foldsFano}
			\begin{tabular}{ccccccc} \toprule
				& $X$ & $E$ & $(-K_F)^4$ & $\chi(\Omega^1_F)$ & $\chi(\Omega^2_F)$ & $\hhh^0(-K_F)$ 
				\\ \midrule
				$\overline{Y_2}\subset \wedge^3\mathbb{C}^6$ & $\Gr(3,6)$ & ${\cal U}_X^* \oplus 3{\cal O}_X$ & $63$ & $-2$ & $21$ & $19$ 
				\\	\midrule
				$\overline{Y_2}\subset \wedge^3\mathbb{C}^6$ & $\IGr(2,7)$ & ${\cal Q}_X\oplus {\cal O}_X$ & $69$ & $-4$ & $26$ & $20$ 
				\\	\midrule
				$\overline{Y_2}\subset \wedge^3\mathbb{C}^6$ & $\IGr(2,7)$ & ${\cal U}_X^* \oplus 4{\cal O}_X$ & $47$ & $-7$ & $54$ & $16$ 
				\\	\midrule
				(3) $\PP^3, SL_4$ & $\QQ^{13}$ & $4{\cal O}_X$ & $40$ & $-18$ & $114$ & $15$
				\\	\midrule
				(7) $\PP^4, SL_5$ & $\PP^{20}$ & $5{\cal O}_X$ & $70$ & $-6$ & $46$ & $21$
				\\ \bottomrule
			\end{tabular}
		\end{small}
	\end{center}
\end{table}

\appendix
\section{Computation of Hodge numbers}
\label{appendix}
This appendix is devoted to explaining how we computed the Hodge numbers of some of the varieties we found as degeneracy loci. In particular, we deal with the case of smooth $\overline{Y_2}$-degeneracy loci studied in Section \ref{partDecForms}. We use standard techniques, such as the Koszul complex and the Leray spectral sequence, to reduce to the computation of cohomologies on the base variety $X$.

As our varieties are smooth, they are isomorphic to their resolutions $\zero(\tilde{s})\subset \PP(E)$. This is just the zero locus of a section of the bundle $Q_W$; hence, the Koszul complex
\[
0\to \wedge^{10}(Q_W^*)\to \dotso \to \wedge^{1}(Q_W^*)\to {\cal O}_{\PP(E)}\to {\cal O}_{\zero(\tilde{s})}\to 0
\]
gives a resolution of ${\cal O}_{\zero(\tilde{s})}$, so it can be used to compute the cohomology of the restriction to $\zero(\tilde{s})$ of a vector bundle on $\PP(E)$. What we need, for example for threefolds, is the cohomology of ${\cal O}_{\zero(\tilde{s})}$ and of $\Omega^1_{\zero(\tilde{s})}$. This last bundle is not the restriction of a bundle on $\PP(E)$, but its cohomology can be recovered by using the (co)normal sequence:
\[
0\to (Q_W^*)|_{\zero(\tilde{s})} \to (\Omega^1_{\PP(E)})|_{\zero(\tilde{s})} \to \Omega^1_{\zero(\tilde{s})}\to 0 \,\,\, .
\]
Therefore, we want to compute
\begin{equation}
	\label{whatwewant}
	\HHH^j(\PP(E), \wedge^i(Q_W^*)\otimes {\cal G}) \qquad \mbox{ for }{\cal G}={\cal O}_{\PP(E)}, Q_W^*,\Omega^1_{\PP(E)}.
\end{equation}
With some chance, this will be enough to determine the desired cohomology groups. To work directly on $X$, we can make use of Leray spectral sequence (see e.g.\ \cite{Voisin02}):

\begin{theorem}[Leray]
	Let $\phi:Z\to X$ be a continuous map between two topological spaces. For every sheaf ${\cal F}$ over $Z$, there exists a canonical filtration on $\HHH^q(Z, {\cal F})$ which is the limit object of a spectral sequence 
	\[
	E^{p,q}_{r} \Rightarrow \HHH^{p+q}(Z,{\cal F})\,\,\,.
	\] 
	The spectral sequence is canonically starting from $E_2$, whose terms are
	\[
	E^{p,q}_2 = \HHH^{p}(X,\RRR^q\phi_*{\cal F}).
	\]
\end{theorem}

Applying the theorem to $\theta: \PP(E)\to X$, we are led to find the cohomology groups $\HHH^p(X, \RRR^q\theta_*(\wedge^i(Q_W^*)\otimes {\cal G}))$. This is not hard, as shown below. It should be noted that it is not clear a priori if the spectral sequence degenerates at $E_2^{p,q}$. However, by the definition of $E_r^{p,q}$,

\[
E_r^{p,q}\to E_r^{p+r,q-r+1} \mbox{ is zero } \quad \Rightarrow \quad E_{r+1}^{p,q}=E_r^{p,q}.
\]
Therefore, if
\begin{equation}
	\label{spectraldeg}
	E_2^{p,q}\to E_2^{p+r,q-r+1} \mbox{ is zero }\forall r\geq 2 \,\, ,
\end{equation}
then $E_{\infty}^{p,q}=E_2^{p,q}$.

As for ${\cal G}=\Omega^1_{\PP(E)}$, it is convenient to work with $\theta^*\Omega^1_{X}$ and $\Omega^1_{\PP(E)/X}$ instead and consider the exact sequence
\[
\xymatrix{
	0 \ar[r] & \theta^*\Omega^1_{X} \ar[r] & \Omega^1_{\PP(E)} \ar[r] & \Omega^1_{\PP(E)/X} \ar[r] & 0}
\]
where the first map is the dual of $d\theta$. Indeed, by the projection formula for the push-forward, 
\[
\RRR^q\theta_*(\wedge^i(Q_W^*)\otimes \theta^* (\Omega^1_{X}))=\RRR^q\theta_*(\wedge^i(Q_W^*))\otimes \Omega^1_{X} \,\,\, .
\]

Moreover, the relative cotangent bundle of a projective bundle is well understood, as $\Omega^1_{\PP(E)/X}\cong {\cal U}\otimes \cQ^*$.

Let 
$\tilde{\cal G}$ stand for 
$\wedge^i(Q_W^*)\otimes {\cal G}$. We want to apply 
$\RRR^q\theta_*(\cdot)$ to it. In all the cases needed, the bundle $\tilde{\cal G}$ is the relative version of a homogeneous bundle over $\PP(V_6)$, say 
$\tilde{G}$, i.e.\ 
$\tilde{\cal G}\cong \cE_{\tilde{G}}$. Moreover,  we can compute the stalk of $\tilde{\cal G}$ on every point $x\in X$ by the formula 
\[
\RRR^q\theta_*(\tilde{\cal G})_x=\HHH^q(\theta^{-1}(x), \tilde{\cal G}|_{\theta^{-1}(x)})\cong \HHH^q(\PP(V_6),  \tilde{G})\,\,\, .
\]
This is given by Bott's Theorem \cite{Bott57} as a Schur functor applied to $V_6^*$, say 
$\HHH^q(\PP(V_6),  \tilde{G})\cong S_{\lambda_1,\dotsc,\lambda_6}V_6^*$. In the relative case, we get:
\begin{equation*}
	\RRR^q\theta_*(\tilde{\cal G})=\RRR^q\theta_*(\cE_{\tilde{G}})=\cE_{\HHH^q(\PP(V_6),  \tilde{G})}\cong S_{\lambda_1,\dotsc,\lambda_6}E^*.
\end{equation*}
As an example, if $\tilde{\cal G}={\cal U}^*$, $\RRR^0\theta_*(\tilde{\cal G})=S_{1,0,0,0,0,0}E^*=E^*$ and the other push-forwards vanish. 

In the end, we obtain the cohomologies on $\PP(E)$ in terms of the cohomologies of certain $S_{\lambda_1,\dotsc,\lambda_6}E^*$ on $X$. For any fixed pair $(i,q)$, the Schur functor $S_\lambda$ associated to $\RRR^q\theta_*(\wedge^i(Q_W^*)\otimes {\cal G})$ does not depend on $X$ or $E$; we collect in Tables \ref{plethysms1} and \ref{plethysms2} the corresponding $\lambda$ for each choice of $\cal G$.

Finally, Bott's Theorem yields $\HHH^p(X,S_\lambda E^*)$. Notice that $S_\lambda E^*$ is not irreducible in general, so some plethysm is needed; we used the computer algebra software \texttt{LiE} (\cite{Lie}) to obtain a decomposition in irreducible homogeneous bundles. As it turns out, in all our cases condition \eqref{spectraldeg} is satisfied, i.e.\ the Leray spectral sequence degenerates at $r=2$. Therefore, these computations are enough to recover the cohomology groups \eqref{whatwewant} of the terms of the Koszul complexes on $\PP(E)$.

{\small
	\begin{table}[h!bt]
		\begin{center}
			\caption{Partitions associated to the push-forward of the bundles on $\PP(E)$; $(\lambda_1,\dotsc,\lambda_6)$ corresponds to $S_{\lambda_1,\dotsc,\lambda_6}E^*$.}
			\label{plethysms1}
			\begin{tabular}{cccc} \toprule
				$i$ & $q$ & $\RRR^q\theta_*(\wedge^iQ_W^*)$ & $\RRR^q\theta_*(\wedge^iQ_W^*\otimes Q_W^*)$ 
				\\ \midrule
				$0$ & $0$ & $(0,0,0,0,0,0)$ & 
				\\ \midrule
				$1$ & $1$ &  & $(1,1,1,1,1,1)$ 
				\\ \midrule
				$2$ & $2$ &  & $(2,2,2,2,1,0)+(2,2,2,1,1,1)$ 
				\\ \midrule
				$3$ & $2$ & $(2,2,2,1,1,1)$ & $(2,2,2,2,2,2)+(3,3,2,2,1,1)+2\times (3,2,2,2,2,1) $ 
				\\ \midrule
				$4$ & $2$ & $(3,2,2,2,2,1)$ & $(4,3,2,2,2,2)$ 
				\\ \midrule
				$4$ & $3$ &  & $(3,3,3,3,3,0)$ 
				\\ \midrule
				$5$ & $3$ &  & $(4,4,3,3,3,1)+(5,3,3,3,2,2)+2\times (4,3,3,3,3,2)$ 
				\\ \midrule
				$6$ & $3$ & $(4,3,3,3,3,2)$ &$\begin{array}{l}
				(4,4,4,3,3,3)+(5,4,4,3,3,2)+(5,4,3,3,3,3) \\
				{}+(6,3,3,3,3,3)
				\end{array}$
				\\ \midrule
				$7$ & $3$ & $(4,4,4,3,3,3)$ & $(5,5,5,3,3,3)$ 
				\\ \midrule
				$7$ & $4$ &  & $(5,4,4,4,4,3)$
				\\ \midrule
				$8$ & $4$ &  & $(5,5,5,4,4,4)+(6,5,4,4,4,4)$ 
				\\ \midrule
				$9$ & $5$ &  & $(5,5,5,5,5,5)$
				\\ \midrule
				$10$ & $5$ & $(5,5,5,5,5,5)$ & $(6,6,6,5,5,5)$
				\\ \bottomrule
			\end{tabular}
		\end{center}
	\end{table}
}

{\small
	\begin{table}[h!bt]
		\begin{center}
			\caption{Partitions associated to the push-forward of the bundles on $\PP(E)$; $(\lambda_1,\dotsc,\lambda_6)$ corresponds to $S_{\lambda_1,\dotsc,\lambda_6}E^*$.}
			\label{plethysms2}
			\begin{tabular}{ccc} \toprule
				$i$ & $q$ & $\RRR^q\theta_*(\wedge^iQ_W^*\otimes ({\cal U_{\PP(E)}}\otimes \cQ_{\PP(E)}^*))$
				\\ \midrule
				$0$ & $1$ & $(0,0,0,0,0,0)$
				\\ \midrule
				$2$ & $2$ & $(1,1,1,1,1,1)+(2,1,1,1,1,0)$
				\\ \midrule
				$3$ & $2$ & $(3,2,1,1,1,1)$
				\\ \midrule
				$4$ & $3$ & $(3,2,2,2,2,1)$
				\\ \midrule
				$5$ & $3$ & $(4,3,2,2,2,2)+(5,2,2,2,2,2)$
				\\ \midrule
				$6$ & $4$ & $(4,3,3,3,3,2)$
				\\ \midrule
				$7$ & $4$ & $(4,4,4,3,3,3)+(5,4,3,3,3,3)+(6,3,3,3,3,3)$
				\\ \midrule
				$9$ & $5$ & $(5,5,5,4,4,4)+(6,5,4,4,4,4)$
				\\ \midrule
				$10$ & $5$ & $(6,5,5,5,5,4)$
				\\ \bottomrule
			\end{tabular}
		\end{center}
	\end{table}
}

\section{A Thom--Porteous type formula}
\label{appendixB}

In this appendix we present, for $Y$ the subvariety of partially decomposable three-forms in $\wedge^3 \mathbb{C}^6$, a Thom--Porteous type formula for the fundamental class of an orbital degeneracy locus $D_Y(s)$ of a section $s \in \HHH^0(X,E)$ in terms of the Chern classes of $E$. A formula expressing the Todd class of a four-dimensional $D_Y(s)$ in terms of the Chern classes of $E$ and of the tangent bundle of $X$ is also given.

\begin{proposition}
	Let $s$ be a general section of the globally generated vector bundle $\wedge^3 E$ on a variety $X$ of arbitrary dimension. Let $e_i$ denote the Chern classes of $E$ and $s_\lambda$ its Schur classes. Then the fundamental class of $D_Y(s)$ is
	\begin{equation*}
		[D_Y(s)] =e_1 \left(e_1^4+e_2^2+2e_1e_3-4e_4\right) = s_{(4)}+3s_{(3,1)}+3s_{(2,2)}+6s_{(2,1,1)}.
	\end{equation*}
	\begin{proof}
		The cohomology ring of $\PP(E)$ is an algebra over the cohomology ring of $X$ and it is generated by $H$ with the relation
		\[H^6=-\sum_{i=1}^6 e_iH^{6-i}.\]
		On $\PP(E)$, the class of $\zero(\tilde{s})$ is the class of a zero locus of a general section of $\wedge^3\cQ_{\PP(E)}$; the Chern classes of $\wedge^3\cQ_{\PP(E)}$ can be easily found in terms of the Chern classes of $\cQ_{\PP(E)}$, and a computer-aided computation yields the following expression for the top Chern class:
		\begin{equation}
			\label{ctopctop}
			\begin{array}{r@{}l}
				\ctop(\wedge^3\cQ_{\PP(E)}) = H^5e_1 & \left(e_1^4+e_2^2+2e_1e_3-4e_4\right) + {}\\
				\rule{0pt}{12pt} H^4e_1 & \left(e_1^5+e_1^3e_2+2e_1e_2^2+e_1^2e_3-e_2e_3-6e_1e_4+2e_5\right) + {}\\
				\rule{0pt}{12pt} H^3e_1 & \left(2e_1^4e_2+2e_1^2e_2^2+e_2^3-e_1^3e_3-e_1e_2e_3+{} \right. \\
				\rule{0pt}{12pt} & \multicolumn{1}{r}{\left. e_3^2-4e_1^2e_4-4e_2e_4+4e_1e_5-4e_6\right) + {}}\\
				\rule{0pt}{12pt} H^2e_1 & \left(2e_1^3e_2^2+e_1e_2^3+e_1^4e_3-3e_1^2e_2e_3+3e_1e_3^2-3e_1^3e_4+{}\right.\\
				\rule{0pt}{12pt} & \multicolumn{1}{r}{\left.-3e_1e_2e_4-2e_3e_4+3e_1^2e_5+e_2e_5-8e_1e_6\right) + {}}\\
				\rule{0pt}{12pt} He_1 & \left(e_1^2e_2^3 + e_1^3e_2e_3 - e_1e_2^2e_3 - e_1^2e_3^2-4e_1^2e_2e_4+e_2^2e_4+ {}\right.\\
				\rule{0pt}{12pt} & \multicolumn{1}{r}{\left.5e_1e_3e_4-4e_4^2+e_1^3e_5+e_3e_5-6e_1^2e_6-2e_2e_6\right)+{}}\\
				\rule{0pt}{12pt} e_6 & \left(-3e_1^4-e_2^2-4e_1e_3+4e_4\right)+{}\\
				\rule{0pt}{12pt} e_5 & \left(e_1^5-e_1^3e_2+3e_1^2e_3+e_2e_3-2e_1e_4-e_5\right)+{}\\
				\rule{0pt}{12pt} e_4 & \left(-e_1^4e_2+e_1^3e_3+e_1e_2e_3-e_3^2-e_1^2e_4\right)+{}\\
				\rule{0pt}{12pt} e_3 & \left( e_1^3e_2^2-2e_1^2e_2e_3+e_1e_3^2\right).
			\end{array}
		\end{equation}
		
		Let $\theta:\PP(E) \rightarrow X$ be the usual projection. The push-forward $\theta_*(H^i)$ is the zero class for $i < 5$, hence the class of $D_Y(s)$ is given by the coefficient of $H^5$ in \eqref{ctopctop}. An easy computation leads to the expression in terms of the Schur classes of $E$ (see e.g.\ \cite{Fulton98}). 
	\end{proof}
\end{proposition}

For any variety $Z$, the Hirzebruch--Riemann--Roch Theorem yields
\[
\chi(\cO_{Z})=
\int_{Z} \td(Z),
\]  
being $\td(Z)$ the Todd class of the tangent bundle to $Z$.
With a little more effort we are able to express the Todd class of $D_Y(s)$ in terms of the Chern classes of $E$ and of the tangent bundle of $X$. In the following formula we write an explicit expression for fourfolds.

\begin{formula}
	\label{formulab2}
	Let $D_Y(s)$ have dimension four. Let $e_i$ and $t_i$ denote the Chern classes of $E$ and of the tangent bundle of $X$ respectively. Then
	\begin{equation}
		\label{todddys}
		\begin{array}{r@{}l}
			\td(D_Y(s)) = e_1e_6 & \left( \tfrac{601}{180} e_1^2 - \tfrac{1}{12}e_2 - \tfrac{5}{4} e_1t_1 + \tfrac{1}{12} t_1^2 + \tfrac{1}{12} t_2\right) + {}\\
			\rule{0pt}{12pt}e_1e_5 & \left( -\tfrac{101}{180}e_1^3+\tfrac{11}{360} e_1e_2 -
			\tfrac{1}{40} e_3 + \tfrac{5}{24} e_1^2t_1 - \tfrac{1}{72} e_1t_1^2 - \tfrac{1}{72} e_1t_2 \right) + {} \\
			\rule{0pt}{12pt}e_1e_4 & \left(-\tfrac{311}{36} e_1^4 + \tfrac{787}{360} e_1^2e_2 - \tfrac{1}{18} e_2^2 - \tfrac{1}{72} e_1 e_3 + \tfrac{145}{24} e_1^3t_1 - {} \right.\\
			\rule{0pt}{12pt}& \multicolumn{1}{c}{\left. \tfrac{5}{6} e_1e_2t_1 - \tfrac{79}{72} e_1^2t_1^2 + \tfrac{1}{18} e_2t_1^2 +  \tfrac{1}{180} t_1^4 - \tfrac{79}{72} e_1^2t_2 + \tfrac{1}{18} e_2t_2 +  {} \right.}\\
			\rule{0pt}{12pt} & \multicolumn{1}{r}{\left. \tfrac{5}{12} e_1t_1t_2 - \tfrac{1}{45} t_1^2t_2 - \tfrac{1}{60} t_2^2 - \tfrac{1}{180} t_1t_3 + \tfrac{1}{180} t_4
				+ \tfrac{1}{45} e_4  \right) + {}} \\
			\rule{0pt}{12pt} e_1e_3 & \left( \tfrac{81}{20} e_1^5 - \tfrac{1}{60} e_1e_2^2 - \tfrac{35}{12} e_1^4t_1 + \tfrac{13}{24} e_1^3 t_1^2 - \tfrac{1}{360} e_1 t_1^4 + \tfrac{13}{24} e_1^3 t_2 - {} \right.\\
			\rule{0pt}{12pt} & \multicolumn{1}{c}{\left. \tfrac{5}{24} e_1^2t_1t_2 + \tfrac{1}{90} e_1t_1^2 t_2 + \tfrac{1}{120} e_1t_2^2 + \tfrac{1}{360} e_1 t_1 t_3 - \tfrac{1}{360} e_1 t_4 \right. - {}}\\
			\rule{0pt}{12pt} & \multicolumn{1}{r}{\left. \tfrac{97}{120} e_1^2e_3 + \tfrac{1}{30} e_2e_3 + \tfrac{5}{16} e_1e_3t_1 - \tfrac{1}{48} e_3t_1^2 - \tfrac{1}{48} e_3t_2 \right) + {}}\\
			\rule{0pt}{12pt} e_1e_2 & \left( \tfrac{81}{40} e_1^4e_2 - \tfrac{35}{24} e_1^3e_2t_1 + \tfrac{13}{48} e_1^2e_2t_1^2 - \tfrac{1}{720} e_2t_1^4 + \tfrac{13}{48} e_1^2e_2t_2 -  {}  \right. \\
			\rule{0pt}{12pt} & \multicolumn{1}{c}{\left. \tfrac{5}{48} e_1 e_2 t_1 t_2 + \tfrac{1}{180} e_2t_1^2t_2 + \tfrac{1}{240} e_2 t_2^2 + \tfrac{1}{720} e_2t_1t_3 - \tfrac{1}{720} e_2 t_4 - {} \right.}\\
			\rule{0pt}{12pt} & \multicolumn{1}{r}{\left. \tfrac{97}{180} e_1^2 e_2^2 + \tfrac{5}{24} e_1e_2^2t_1 - \tfrac{1}{72} e_2^2t_1^2 - \tfrac{1}{72} e_2^2t_2 + \tfrac{1}{80} e_2^3 \right) + {}}\\
			\rule{0pt}{12pt} e_1^5 & \left( -\tfrac{1}{720}t_1^4 + \tfrac{1}{180}t_1^2t_2 + \tfrac{1}{240} t_2^2 + \tfrac{1}{720} t_1t_3 - \tfrac{1}{720} t_4 - {} \right.\\
			\rule{0pt}{12pt} & \multicolumn{1}{r}{\left. \tfrac{5}{48}e_1t_1t_2 +\tfrac{5}{18} e_1^2t_1^2 + \tfrac{5}{18} e_1^2 t_2 - \tfrac{25}{16} e_1^3t_1 + \tfrac{331}{144} e_1^4 \right).}
		\end{array}
	\end{equation}
	\begin{proof}
		We can compute the Todd class of the resolution of singularities $\zero(\tilde{s})$, which is isomorphic to $D_Y(s)$ by hypothesis. Since
		\[
		\td(\zero(\tilde{s}))= \frac{\td(\PP(E))}{\td{(\wedge^3 \cQ_{\PP(E)})}} \ctop(\wedge^3\cQ_{\PP(E)}),
		\]
		we need to compute the Todd classes of the tangent bundle of $\PP(E)$ and of $\wedge^3\cQ_{\PP(E)}$, which can be expressed in terms of the corresponding Chern classes. The Chern polynomial of the tangent bundle of $\PP(E)$ can be found as the product of the Chern polynomials of the relative tangent bundle $\cQ_{\PP(E)}(1)$ and the tangent bundle of $X$.
	\end{proof}
\end{formula}

The formula above holds for a four-dimensional degeneracy locus $D_Y(s)$ inside a nine-dimensional variety $X$. In particular, for $X$ a Fano variety of index 5 with $K_X=(L^*)^5$ and $e_1:=c_1(E)=c_1(L)$, formula \eqref{todddys} with $t_1=5e_1$ yields an expression for the Todd class of a $D_Y(s)$ with trivial canonical bundle.

Suppose that $X$ is Fano of index $i$ with $K_X=(L^*)^i$, and suppose that $6 \leq i \leq 10$. Suppose that $e_1=c_1(L)$; then $D_Y(s)$ turns out to be a Fano variety, as discussed in Section \ref{fanoDegLoci}. In particular \eqref{todddys}, with the substitution $t_1=ie_1$, yields the constant value 1 by the Hirzebruch--Riemann--Roch Theorem. Is there a simple interpretation of Formula \ref{formulab2} which explains this phenomenon?

\medskip\noindent {\sc Problem}. Find a Thom--Porteous type formula for other $G$-invariant subvarieties $Y$ inside a $G$-representation $V$.


\makeatletter
\providecommand\@dotsep{5}
\makeatother
\listoftodos\relax

\end{document}